\documentclass[reqno,12pt]{amsart}
\usepackage[colorlinks=true, linkcolor=blue, citecolor=blue]{hyperref}

\usepackage{amssymb}
\usepackage{amsmath, graphicx, rotating}
\usepackage{color}
\usepackage{soul}
\usepackage[dvipsnames]{xcolor}

\usepackage{ifthen}
\usepackage{xkeyval}
\usepackage{todonotes}
\usepackage[T1]{fontenc}
\usepackage{lmodern}
\usepackage[english]{babel}

\usepackage{ upgreek }
\usepackage{stmaryrd}
\SetSymbolFont{stmry}{bold}{U}{stmry}{m}{n}
\usepackage{amsthm}
\usepackage{float}

\usepackage{ bbm }
\usepackage{ stmaryrd }
\usepackage{ mathrsfs }
\usepackage{ frcursive }
\usepackage{ comment }

\usepackage{pgf, tikz}
\usetikzlibrary{shapes}
\usepackage{varioref}
\usepackage{enumitem}

\definecolor{rouge}{rgb}{0.7,0.00,0.00}
\definecolor{vert}{rgb}{0.00,0.5,0.00}
\definecolor{bleu}{rgb}{0.00,0.00,0.8}
\usepackage[margin=1in]{geometry}
\newtheorem{theorem}{Theorem}[section]
\newtheorem*{theorem*}{Theorem}
\newtheorem{lemma}[theorem]{Lemma}
\newtheorem{definition}[theorem]{Definition}
\newtheorem{corollary}[theorem]{Corollary}

\labelformat{hypothesis}{\textbf{M\kern-0.1mm#1}}

\newtheorem{condition}{Condition}

\newtheorem{conditionA}{A\kern-0.1mm}
\labelformat{conditionA}{\textbf{A\kern-0.1mm#1}}

\theoremstyle{definition}

\newtheorem{remark}[theorem]{Remark}

\def \eref#1{\hbox{(\ref{#1})}}

\numberwithin{equation}{section}

\def\geq{\geqslant}
\def\leq{\leqslant}

\def\RR{\mathbb{R}}
\def\PP{\mathbb{P}}
\def\EE{\mathbb{E}}

\def\de{{\delta}}

\def\de{{\delta}}

\def\vare{{\varepsilon}}
\def \eref#1{\hbox{(\ref{#1})}}

\def\EE{\mathbb{ E}}

\begin{document}

\title[Averaging principle for 2D stochastic Navier-Stokes equations]
{ Averaging principle for two dimensional stochastic Navier-Stokes equations}

\author{Shihu Li}
\curraddr[Shihu Li]{ School of Mathematical Sciences, Nankai
University, Tianjin 300071, China }
\email{shihuli@mail.nankai.edu.cn}

\author{Xiaobin Sun}
\curraddr[Xiaobin Sun]{ School of Mathematics and Statistics, Jiangsu Normal University, Xuzhou, 221116, China}
\email{xbsun@jsnu.edu.cn}

\author{Yingchao Xie}
\curraddr[Yingchao Xie]{ School of Mathematics and Statistics, Jiangsu
Normal University, Xuzhou, 221116, China} \email{ycxie@jsnu.edu.cn}

\author{Ying Zhao}
\curraddr[Ying Zhao]{ School of Mathematics and Statistics, Jiangsu
Normal University, Xuzhou, 221116, China} \email{2020160621@jsnu.edu.cn}

\begin{abstract}
The averaging principle is established for the slow component and the fast component being two dimensional stochastic Navier-Stokes equations
and stochastic reaction-diffusion equations, respectively. The classical Khasminskii approach based on time discretization is used for the proof of the slow component strong convergence to the solution of the corresponding averaged equation under some suitable conditions.
Meanwhile, some powerful techniques are used to overcome the difficulties caused by the nonlinear term and to release the regularity
of the initial value.
\end{abstract}
\date{\today}

\maketitle
\section{Introduction}
In this paper, we shall establish the averaging principle of the following stochastic fast-slow system
on the 2D tours $\mathbb{T}^{2}=\mathbb{R}^2/(2\pi\mathbb{Z})^{2}$:
\begin{equation}\left\{\begin{array}{l}\label{Equation}
\displaystyle
dX^{\varepsilon}_t=\left[\nu\Delta X^{\varepsilon}_t-(X^{\varepsilon}_t\cdot\nabla)X^{\varepsilon}_t
+f(X^{\varepsilon}_t, Y^{\varepsilon}_t)-\nabla p\right]dt+\sigma_1(X^{\varepsilon}_t)d W^{Q_{1}}_{t},\\
\displaystyle dY^{\varepsilon}_t=\frac{1}{\varepsilon}\left[\Delta Y^{\varepsilon}_t+g(X^{\varepsilon}_t, Y^{\varepsilon}_t)\right]dt
+\frac{1}{\sqrt{\varepsilon}}\sigma_2(X^{\varepsilon}_t, Y^{\varepsilon}_t)d W^{Q_{2}}_{t},\\
\displaystyle \nabla\cdot X^{\varepsilon}_t=0, ~\nabla\cdot Y^{\varepsilon}_t=0, \\
X^{\varepsilon}_0=x, Y^{\varepsilon}_0=y,\end{array}\right.
\end{equation}
where $\varepsilon>0$ is a small parameter describing the ratio of time scale between the slow component $X^{\varepsilon}_t$
and the fast component $Y^{\varepsilon}_t$,  $\Delta$ is the Laplace operator, $p$ denotes the pressure, $\nu>0$ is the kinematic viscosity,
$f$, $g$, $\sigma_1$ and $\sigma_2$ satisfy some suitable conditions. $\{W^{Q_1}_t\}_{t\geq 0}$ and $\{W^{Q_2}_t\}_{t\geq 0}$
are $L^2(\mathbb{T}^{2}, \RR^2)$-valued mutually independent $Q_1$ and $Q_2$-Wiener processes on complete probability space
$(\Omega,\mathscr{F},\{\mathscr{F}_{t}\}_{t\geq0},\mathbb{P})$.

The averaging principle for multiscale system has a long and rich history, and has wide applications in material sciences, chemistry,
fluids dynamics, biology, ecology, climate dynamics, see, e.g., \cite{BR,WE,HKW,EEJ,MCCTB, WTRY16} and references therein.
Bogoliubov and Mitropolsky \cite{BM} first studied the averaging principle for the deterministic systems. Then Khasminskii \cite{K1}
studied averaging principle for stochastic differential equations (SDEs), see, e.g., \cite{G,L,LSX,XLM,XPW} for further generalization.
Recently, averaging principles for stochastic partial differential equations (SPDEs) have attracted much attention.
For example, Cerrai and Freidlin \cite{CF} proved the averaging principle for a general class of stochastic reaction-diffusion
systems with two time-scales, which has been extended to the more general model in \cite{C1,C2,CL}.
Br\'{e}hier \cite{B1} gave the strong and weak orders in averaging for stochastic evolution equation of parabolic
type with slow and fast time scales. In \cite{FLL}, Fu, Wan and Liu proved the strong averaging principle for stochastic
hyperbolic-parabolic equations with slow and fast time-scales.
For more interesting results on this topic, we refer to \cite{FL,FLWL,WR12,WRD12} and references therein.

However, there are few results on the average principle for SPDEs with highly nonlinear term. Recently, the second author and his cooperators \cite{DSXZ} have established the strong and weak averaging principle for one dimensional stochastic Burgers equation
with additive noise. Averaging principle for stochastic Kuramoto-Sivashinsky equation with a fast oscillation was studied by Gao in \cite{GP}. In this paper, we focus on studying the strong averaging principle for 2D stochastic Navier-Stokes equations with multiplicative noise.
To be more precise, we will prove that
\begin{equation}\label{1.2}
\lim_{\vare\to 0}\mathbb{E} \left( \sup_{t\in[0,T]} |X^{\varepsilon}_t-\bar{X}_t |^{2p} \right)=0, \quad p\geq 1,
\end{equation}
where $\bar{X}_t$ is the solution of the corresponding averaged equation (see equation \eref{1.3} below). The 2D stochastic
Navier-Stokes equations have been studied by many authors,  for instance, we refer to \cite{Daz1,DX,HM1,LR1,MS,SS} and the references therein.

The proof of our main result is based on the Khasminskii discretization introduced in \cite{K1}, which is a powerful skill
to study the averaging principle for different types of systems with two time-scales. More precisely, we split the interval $[0,T]$ into
some subintervals of size $\delta>0$ which depends on $\vare$, and on each interval $[k\delta, (k+1)\delta)]$, $k\geq 0$,
we construct an auxiliary process $(\hat{X}_t^\vare, \hat{Y}_t^\vare)$ which associate with the system
(\ref{Equation}). Then (\ref{1.2}) can be proved by the following two steps. Step 1, due to the highly nonlinear term in
stochastic Navier-Stokes equation, we will use stopping time techniques and control the difference of $X_t^\vare$ and $\bar{X}_t$
before the stopping time, which will be done by  controlling $|X_t^\vare-\hat{X}_t^\vare |$ and $|\hat{X}_t^\vare-\bar{X}_t|$ respectively.
Step 2, after the stopping time term can be estimated by the priori estimates of the solution.

Comparing with some recent works on strong convergence in averaging principle for SPDEs (cf. \cite{B1,FL,FLL,FLWL}),
the main challenge in the research of the strong convergence \eqref{1.2} is the nonlinear term of the Navier-Stokes equation.
Moreover, due to the dimension of space is two and multiplicative noise, the skills used in \cite{DSXZ} don't work in the situation for our case.
In order to overcome the difficulties, we shall deal with the nonlinear term and the multiplicative noise more delicately.

Because of the approach based on time discretization, the H\"{o}lder continuity of time for $X_{t}^{\varepsilon}$ would play an important role in the proof of the average principle usually. To this purpose, the condition of the initial value $x\in H^{\theta}$ (the Sobolev space, see Section 2) for some $\theta>0$ will be assumed usually, for example, see \cite[Proposition 4.4]{C1}, \cite[Lemma 3,4]{DSXZ} and \cite[Proposition 9]{GP}. However in this paper, we would like to stress the initial value $x\in H$, then replace studying the H\"{o}lder continuity of time by proving a weak result relatively (see Lemma \ref{COX} below), and it would be enough to prove our main result. Hence, the techniques used here are very helpful to weaken the regularity of initial value $x$. We also believe that these techniques can be applied to more general framework of SPDEs, which will be stated in our forthcoming paper.

The rest of the paper is organized as follows. In Section \ref{Sec Main Result}, 
under some suitable assumptions, we formulate our main result.
Section \ref{Sec Proof of Thm1} is devoted to proving our main result. In the Appendix \ref{Sec appendix},  we
give some properties of the nonlinear term and the proof of the well-posedness of our system.

Throughout the paper, $C$, $C_p$ and $C_{R,T}$ will denote positive constants which may change from line to line, where $C_p$ depends on
$p$, $C_{R,T}$ depends on $R, T$.

\section{Notations and main results} \label{Sec Main Result}

For $p\geq1$, let $L^p(\mathbb{T}^{2}, \RR^2)$ be the space of $p$-th power integrable $\RR^2$-valued functions on torus $\mathbb{T}^{2}$ and $|\cdot|_{L^p}$ be the usual norm. For $k\in\mathbb{N}$, $W^{k,2}(\mathbb{T}^2)$ is the Sobolev space
of all functions in $L^2(\mathbb{T}^{2}, \RR^2)$ whose differentials belong to $L^2(\mathbb{T}^{2}, \RR^2)$ up to the order $k$. Let $L^2_0(\mathbb{T}^{2},\RR^2)$ be the space of  square-integrable $\RR^2$-valued functions on the torus with vanishing mean, i.e.,
$$L^2_0(\mathbb{T}^{2},\RR^2):=\left\{u\in L^2(\mathbb{T}^{2},\mathbb{R}^{2}):~\int_{\mathbb{T}^{2}}u(\xi)d\xi=0\right\}.$$
We consider a Hilbert space $H$ which is a closed subspace of $L^2_0(\mathbb{T}^{2}, \RR^2)$, defined by
$$H:=\left\{u\in L^2_0(\mathbb{T}^{2}, \RR^2):~\nabla\cdot u=0\right\}.$$
The space $H$ is endowed with the inner product and the norm on $L^2(\mathbb{T}^{2}, \RR^2)$, which denoted by $\langle\cdot,\cdot\rangle$ and $|\cdot|$ respectively.

We shall fix an orthonormal basis $\{e_k\}_{k\geq1}$ of $H$ consisting of the eigenvectors
of $\Delta$, i.e.,
$$
\Delta e_k=-\lambda_k e_k,
$$
where $0<\lambda_1\leq\lambda_2\leq\cdots\leq\lambda_k\uparrow\infty$.
Moreover, we put $\mathscr{D}(A):=W^{2,2}(\mathbb{T}^2)\cap H$, and define the linear operator
$$Au:=P_{H}\Delta u,\quad u\in \mathscr{D}(A),$$
where $\emph{P}_H$ is the Helmholtz-Leray projector from
$L^2_0(\mathbb{T}^{2}, \RR^2)$ onto $H$. Furthermore, in our case
it is known that $A=\Delta$ due to the periodic boundary condition
(see, e.g., \cite{FMRT}). 
For simplicity, we also assume the viscosity constant $\nu=1$ in this paper.

For any $s\in\RR$, we define
 $$H^s:=\mathscr{D}((-A)^{s/2}):=\left\{u=\sum_{k}u_ke_k: u_k=\langle u,e_k\rangle\in \mathbb{R},~\sum_k\lambda_k^{s}u_k^2<\infty\right\},$$
 and
 $$(-A)^{s/2}u:=\sum_k\lambda_k^{s/2} u_ke_k,~~u\in\mathscr{D}((-A)^{s/2}),$$
with the associated norm
\begin{eqnarray*}
\|u\|_{s}:=|(-A)^{s/2}u|=\sqrt{\sum_k\lambda_k^{s} u^2_k}.
\end{eqnarray*}
It is easy to see $H^{0}=H$ and  $H^{-s}$ be the dual space of $H^s$. Notice that  the dual action is also denoted by $\langle\cdot,\cdot\rangle$ without confusion.

It is well known that $A$ is the infinitesimal generator of a strongly continuous semigroup $\{e^{tA}\}_{t\geq 0}$. For $\theta\geq 0$ and $x\in H$, there exits $C$ depends on $\theta$ such that
\begin{eqnarray}
\|e^{tA}x\|_\theta\leq C t^{-\frac{\theta}{2}}|x|.\label{SP}
\end{eqnarray}

Define the bilinear operator
$$B(u,v):H^1\times{H}^1\longrightarrow {H}^{-1}, B(u,v)=P_{H}\big((u\cdot\nabla)v\big)$$
and the trilinear operator
$$b(u,v,w)=\langle B(u,v),w\rangle=\sum_{i,j=1}^{2}\int_{\mathbb{T}^{2}}u_{i}(\xi)\frac{\partial v_{j}(\xi)}
{\partial {\xi}_{i}}w_{j}(\xi)d\xi, ~~~\text{for}~ u,v,w\in {H}^1.$$
Moreover, it is convenient to put $B(u)=B(u,u)$, for $u\in {H}^1$. The related properties of operators $b$ and $B$ are listed in the appendix.

Now, by applying the operator $P_H$ to the first equation of the system
\eref{Equation},  we remove the pressure term and consider the
following abstract stochastic evolution equations:
\begin{equation}\left\{\begin{array}{l}\label{main equation}
\displaystyle
dX^{\varepsilon}_t=\left[A X^{\varepsilon}_t-B(X^{\varepsilon}_t)+f(X^{\varepsilon}_t, Y^{\varepsilon}_t)\right]dt
+\sigma_1(X^{\varepsilon}_t)d W^{Q_{1}}_{t},\\
\displaystyle
dY^{\varepsilon}_t=\frac{1}{\varepsilon}\left[A Y^{\varepsilon}_t+g(X^{\varepsilon}_t, Y^{\varepsilon}_t)\right]dt
+\frac{1}{\sqrt{\varepsilon}}\sigma_2(X^{\varepsilon}_t, Y^{\varepsilon}_t)d W^{Q_{2}}_{t},\\
X^{\varepsilon}_0=x, Y^{\varepsilon}_0=y.\end{array}\right.
\end{equation}
Here $W^{Q_i}_t$ ($i=1,2$) are $H$-valued $Q_i$-Wiener process and $Q_i$ is a positive symmetric, trace class operate on $H$.

Put $U_i=Q_i^{1/2}H$ the Hilbert space with the inner product
$$\langle u,v\rangle_{U_i}=\langle Q_i^{-1/2}u,Q_i^{-1/2}v\rangle,~~u,v\in U_i$$
with the norm $|\cdot|_{U_i}=\sqrt{\langle\cdot,\cdot\rangle_{U_i}}$.
Let $\mathcal{L}_{Q_i}(U_i,H)$ be the space of linear operates $S:U_i\to H$ such that
$SQ_i^{1/2}$ is a Hilbert-Schmidit operator on $H$.
The norm on $\mathcal{L}_{Q_i}(U_i,H)$ is defined by
$$|S|^2_{\mathcal{L}_{Q_i}}=\text{Tr}(SQ_iS^*):=\sum_{k\geq1}|SQ^{1/2}_i e_k|^2,$$
where $S^*$ is the adjoint operator of $S$. In the following, we always assume that $W^{Q_1}_t$ and $W^{Q_2}_t$ are independent.

\medskip
We assume that $f, g: H\times H\to H$, $\sigma_1: H\to \mathcal L_{Q_1}(U_1; H)$ and $\sigma_2: H\times H\to\mathcal L_{Q_2}(U_2; H)$
satisfy the following conditions:


\begin{conditionA}\label{A1}
$f$, $g$, $\sigma_1$ and $\sigma_2$ are Lipschitz continuous, i.e., there exist some positive constants $L_{g}, L_{\sigma_2}$ and $C$ such that for any $x_1,x_2,y_1,y_2\in H$,
\begin{align*}
|f(x_1, y_1)-f(x_2, y_2)|\leq C\left(|x_1-x_2| + |y_1-y_2|\right);
\end{align*}
\begin{align*}
|g(x_1, y_1)-g(x_2, y_2)|\leq C|x_1-x_2| + L_{g}|y_1-y_2|;
\end{align*}
\begin{align*}
\left|\sigma_1(x_1)-\sigma_1(x_2)\right|_{\mathcal{L}_{Q_1}}\leq C|x_1-x_2|;
\end{align*}
\begin{align*}
\left|\sigma_2(x_1, y_1)-\sigma_2(x_2, y_2)\right|_{\mathcal{L}_{Q_2}} \leq C|x_1-x_2| + L_{\sigma_2}|y_1-y_2|.
\end{align*}
\end{conditionA}


\begin{conditionA}\label{A3}
There exists a constant  $\zeta\in(0,1)$, such that
\begin{align*}
|\sigma_2(x, y)|_{\mathcal{L}_{Q_2}}\leq C(1+|x| + |y|^{\zeta}) ~~~ \text{for any}~ x, y\in H.
\end{align*}
\end{conditionA}

\begin{conditionA}\label{A4} The smallest eigenvalue $\lambda_1$ of $-\Delta$ and the Lispchitz constants $L_{g}$, $L_{\sigma_2}$ satisfy
$$
2\lambda_{1}-2L_{g}-L^2_{\sigma_2}>0.
$$
\end{conditionA}

\begin{remark} The condition \ref{A1} ensures the existence and uniqueness of the solution of system \eref{main equation}. The condition \ref{A3} is used to prove all the moments of the solution $(X^{\vare}_t, Y^{\vare}_t)$ are finite, which could be removed if we assume the Lispchitz constant $L_{\sigma_2}$ is sufficiently small. The condition \ref{A4} is called the dissipative condition, which can guarantee that there exits a unique invariant measure for frozen equation and the exponential ergodicity holds.
\end{remark}

\medskip
Now, we recall the following definition.

\begin{definition}\label{S.S.}
For any initial value $x,y\in H$. The system \eqref{main equation} has a weak solution if there exist
$X^{\varepsilon} \in C([0,T]; {H}) \cap L^2(0, T; {H}^1)$ and  $Y^{\varepsilon} \in C([0,T]; H) \cap L^2(0, T; H^1)$, $\PP$-a.s.,
such that, for any $t\in[0,T]$ and $\phi, \varphi\in \mathscr{D}(A)$, the following identity hold
 \begin{eqnarray*}\label{def1}
 \langle X^{\varepsilon}_t,\phi\rangle=\!\!\!\!\!\!\!\!&&\langle x,\phi\rangle+\int_0^t\langle
 X^{\varepsilon}_s,A\phi\rangle ds-\int_0^t\langle B(X^{\varepsilon}_s),\phi\rangle ds\nonumber\\
&& +\int_0^t\langle f(X^{\varepsilon}_s, Y^{\varepsilon}_s),\phi\rangle ds+\int_0^t\langle \sigma_1(X^{\varepsilon}_s)dW^{Q_1}_s,
 \phi\rangle, \quad \PP\text{-a.s.}
\end{eqnarray*}
and
 \begin{eqnarray*}\label{def2}
\langle Y^{\varepsilon}_t,\varphi\rangle=\!\!\!\!\!\!\!\!&&\langle y,\varphi\rangle+\frac{1}{\varepsilon}
\int_0^t\langle Y^{\varepsilon}_s,A\varphi\rangle ds+\frac{1}{\varepsilon}\int_0^t\langle g(X^{\varepsilon}_s,Y^{\varepsilon}_s),
\varphi\rangle ds\nonumber\\
 && +\frac{1}{\sqrt{\varepsilon}}\int_0^t\langle \sigma_2(X^{\varepsilon}_s,Y^{\varepsilon}_s)dW^{Q_2}_s,\varphi\rangle, \quad \PP\text{-a.s.}.
\end{eqnarray*}
\end{definition}

Based on the local-monotonicity method, we have the following well-posedness result.

\begin{theorem}\label{Th1}
Assume the condition \ref{A1} holds. Then for initial value $x, y\in H$, the system \eqref{main equation} has a unique weak solution,
denoted by $(X^{\varepsilon},Y^{\varepsilon})$.
\end{theorem}

Note that this solution is a strong one in the probabilistically meaning.
Using the local-monotonicity method one can prove the existence of weak solution for 2D Navier-Stokes equations.
For completeness, we will prove Theorem \ref{Th1} in the appendix .

\medskip
Now, we state our main result.
\begin{theorem}\label{main result 1}
Assume that the conditions \ref{A1}-\ref{A4} hold. Then for $x, y\in H$, $p\geq1$ and $T>0$, we have
\begin{align}
\lim_{\vare\rightarrow 0}\mathbb{E} \left(\sup_{t\in[0,T]}|X_{t}^{\vare}-\bar{X}_{t}|^{2p} \right)=0,\label{2.2}
\end{align}
where $\bar{X}_t$ is the solution of the corresponding averaged equation:
\begin{equation}\left\{\begin{array}{l}
\displaystyle d\bar{X}_{t}=A\bar{X}_{t}dt-B(\bar{X}_t)dt+\bar{f}(\bar{X}_{t})dt+\sigma_1(\bar{X}_t)d W^{Q_{1}}_{t},\\
\bar{X}_{0}=x,\end{array}\right. \label{1.3}
\end{equation}
with the average $\bar{f}(x)=\int_{H}f(x,y)\mu^{x}(dy)$. $\mu^{x}$ is the unique invariant measure of the frozen equation
\begin{eqnarray*}
\left\{ \begin{aligned}
&dY_{t}=[AY_{t}+g(x,Y_{t})]dt+\sigma_2(x,Y_t)d\bar{W}_{t}^{Q_{2}},\\
&Y_{0}=y,
\end{aligned} \right.
\end{eqnarray*}
$\bar{W}_{t}^{Q_{2}}$ is a $Q_2$-Wiener process, which is independent of $W^{Q_1}_t$ and $W^{Q_2}_t$.
\end{theorem}

\section{Proof of Theorem \ref{main result 1}} \label{Sec Proof of Thm1}

In this section, we are devoted to proving Theorem \ref{main result 1}.
The proof consists of the following several steps.
In the subsection 3.1, we first give some priori estimates of the solution $(X^{\varepsilon}_t, Y^{\varepsilon}_t)$
to the system \eref{main equation}, then prove a weaker result than the H\"{o}lder continuity of time for $X_{t}^{\varepsilon}$.
In the subsection 3.2, following the idea inspired by Khasminskii in \cite{K1},
we introduce an auxiliary process $(\hat{X}_{t}^{\varepsilon},\hat{Y}_{t}^{\varepsilon})$ and also give its uniform bounds.
Meanwhile, making use of the skills of stopping time, we also deduce an estimate of the (difference) process $X^{\varepsilon}_t-\hat{X}_{t}^{\varepsilon}$ when time $t$ is before the stopping time.
In the subsection 3.3, based on the exponential ergodicity of frozen equation, we give the control of the difference process  $\hat{X}^{\varepsilon}_t-\bar{X}_{t}$ when time $t$ is before the stopping time. Finally, we will use the priori estimates of the solution to control the term of time $t$ after the stopping time.  Note that we always assume conditions \ref{A1}-\ref{A4} hold in this section.

\subsection{Some priori estimates of \texorpdfstring{$(X^{\varepsilon}_t, Y^{\varepsilon}_t)$}{Lg}}
At first, we prove uniform bounds with respect to $\vare\in (0,1)$ for $p$-moment of the solutions $(X_{t}^{\varepsilon}, Y_{t}^{\vare})$ to the system \eref{main equation}.
\begin{lemma} \label{PMY} For any $x,y\in H$, $T>0$, $p\geq1$ and $\vare\in(0,1)$, there exists a constant $C_{p,T}>0$ such that for any
\begin{align}
\mathbb{E}\left(\sup_{t\in[0,T]}|X_{t}^{\vare}|^{2p}\right)+\EE\left(\int^T_0|X_{t}^{\vare}|^{2p-2}\|X_{t}^{\vare}\|^2_1dt\right)\leq C_{p,T}\left(1+|x|^{2p}+|y|^{2p}\right)\label{F3.1}
\end{align}
and
\begin{align}
\sup_{t\in[0, T]}\mathbb{E}|Y_{t}^{\varepsilon}|^{2p}\leq C_{p,T}\left(1+|x|^{2p}+|y|^{2p}\right).
\end{align}
\end{lemma}
\begin{proof}
According to It\^{o}'s formula, we have
\begin{eqnarray*}\label{ItoFormu 000}
|Y_{t}^{\vare}|^{2}=\!\!\!\!\!\!\!\!&&|y|^{2}+\frac{2}{\varepsilon}\int_{0} ^{t}\langle AY_{s}^{\varepsilon},Y_{s}^{\varepsilon}\rangle ds+\frac{2}{\varepsilon}\int_{0} ^{t}\langle g(X_{s}^{\varepsilon},Y_{s}^{\varepsilon}),Y_{s}^{\varepsilon}\rangle ds \nonumber \\
 \!\!\!\!\!\!\!\!&& + \frac{1}{\varepsilon}\int_{0} ^{t}|\sigma_2(X_{s}^{\varepsilon},Y_{s}^{\varepsilon})|_{\mathcal{L}_{Q_2}}^2ds
 +\frac{2}{\sqrt{\varepsilon}}\int_{0} ^{t}\langle\sigma_2(X_{s}^{\varepsilon},Y_{s}^{\varepsilon})dW^{Q_2}_s,Y_{s}^{\varepsilon}\rangle.
\end{eqnarray*}
Applying It\^{o}'s formula for $g(z)=(z)^{p}$ and $z_t=|Y_{t}^{\vare}|^{2}$, then taking
expectation on both side, we obtain
\begin{eqnarray*}\label{ItoFormu 001}
\mathbb{E}|Y_{t}^{\vare}|^{2p}=\!\!\!\!\!\!\!\!&&|y|^{2p}+\frac{2p}{\varepsilon}\mathbb{E}\left(\int_{0} ^{t}| Y_{s}^{\varepsilon}|^{2p-2}\langle AY_{s}^{\varepsilon},Y_{s}^{\varepsilon}\rangle ds\right)\nonumber \\
 \!\!\!\!\!\!\!\!&& +\frac{2p}{\varepsilon}\mathbb{E}\left(\int_{0} ^{t}| Y_{s}^{\varepsilon}|^{2p-2}\langle g(X_{s}^{\varepsilon},Y_{s}^{\varepsilon}),Y_{s}^{\varepsilon}\rangle ds \right) + \frac{p}{\varepsilon}\mathbb{E}\left(\int_{0} ^{t}|Y_{s}^{\varepsilon}|^{2p-2}|\sigma_2(X_{s}^{\varepsilon},Y_{s}^{\varepsilon})|_{\mathcal{L}_{Q_2}}^2ds\right)
 \nonumber \\
 \!\!\!\!\!\!\!\!&&+\frac{2p(p-1)}{\varepsilon}\mathbb{E}\left(\int_{0} ^{t}|Y_{s}^{\varepsilon}|^{2p-4}|\sigma_2(X_{s}^{\varepsilon},
 Y_{s}^{\varepsilon})^*Y^\varepsilon_s|_{U_2}^2ds\right).\nonumber
\end{eqnarray*}
Notice that $\langle Ax, x\rangle=-\|x\|^2_1\leq -\lambda_{1}|x|^2$ and by conditions \ref{A1} and \ref{A3},  there exists a constant $\gamma>0$ such that

\begin{eqnarray*}
\frac{d}{dt}\mathbb{E}|Y_{t}^{\vare}|^{2p}=\!\!\!\!\!\!\!\!&&-\frac{2p}{\varepsilon}\mathbb{E}\left(| Y_{t}^{\varepsilon}|^{2p-2}\|Y_{t}^{\varepsilon}\|_{1}^2 \right) +\frac{2p}{\varepsilon}\mathbb{E}\left[| Y_{t}^{\varepsilon}|^{2p-2}\langle g(X_{t}^{\varepsilon},Y_{t}^{\varepsilon}),Y_{t}^{\varepsilon}\rangle\right] \nonumber \\
 \!\!\!\!\!\!\!\!&& + \frac{p}{\varepsilon}\mathbb{E}\left[|Y_{t}^{\varepsilon}|^{2p-2}
|\sigma_2(X_{t}^{\varepsilon},Y_{t}^{\varepsilon})|_{\mathcal{L}_{Q_2}}^2\right]
 +\frac{2p(p-1)}{\varepsilon}\mathbb{E}\left[|Y_{t}^{\varepsilon}|^{2p-4}
 |\sigma_2(X_{t}^{\varepsilon},Y_{t}^{\varepsilon})^*Y^\varepsilon_t|_{U_2}^2\right]\nonumber \\
 \leq\!\!\!\!\!\!\!\!&&-\frac{2p\lambda_{1}}{\varepsilon}\mathbb{E}|Y_{t}^
 {\varepsilon}|^{2p}+
\frac{2p}{\varepsilon}\mathbb{E}\Big[|Y_{t}^{\varepsilon}|^{2p-2}\left(C|Y_{t}^{\varepsilon}|+L_{g}|X_{t}^{\varepsilon}|\cdot|Y_{t}^{\varepsilon}|+L_g |Y_{t}^{\varepsilon}|^2)\right) \Big]\nonumber \\
 \!\!\!\!\!\!\!\!&& +
 \frac{C_p}{\varepsilon}\mathbb{E}\left[|Y_{t}^{\varepsilon}|^{2p-2}
 L_{\sigma_2}^2(C+|X_{t}^{\varepsilon}|^2+|Y_{t}^{\varepsilon}|^{2\zeta})\right]
\nonumber \\
\leq\!\!\!\!\!\!\!\!&&-\frac{p\gamma}{\varepsilon}\mathbb{E}|Y_{t}^{\varepsilon}
|^{2p}+\frac{C_{p}}{\varepsilon}\EE|X_{t}^{\varepsilon}|^{2p}
+\frac{C_{p}}{\varepsilon},\label{4.4.2}
\end{eqnarray*}
where the last inequality comes from the Young's inequality. Hence, by comparison theorem, it is easy to see that
\begin{eqnarray}
\mathbb{E}|Y_{t}^{\varepsilon}|^{2p}\leq\!\!\!\!\!\!\!\!&&|y|^{2p}e^{-\frac{p\gamma}{\varepsilon}t}+\frac{C_{p}}{\varepsilon}\int^t_0
e^{-\frac{p\gamma}{\varepsilon}(t-s)}\left(1+\EE|X_{s}^{\varepsilon}|^{2p}\right)ds.\label{F3.4}
\end{eqnarray}

On the other hand, using It\^{o}'s formula again, we also have
\begin{eqnarray*}
|X_{t}^{\varepsilon}|^{2p}=\!\!\!\!\!\!\!\!&&|x|^{2p}+2p\int_{0} ^{t}| X_{s}^{\varepsilon}|^{2p-2}\langle AX_{s}^{\varepsilon},X_{s}^{\varepsilon}\rangle ds-2p\int_{0} ^{t}| X_{s}^{\varepsilon}|^{2p-2}\langle B(X_{s}^{\varepsilon}),X_{s}^{\varepsilon}\rangle ds \\
 \!\!\!\!\!\!\!\!&&+2p\int_{0} ^{t}|X_{s}^{\varepsilon}|^{2p-2}\langle f(X_{s}^{\varepsilon},Y_{s}^{\varepsilon}),X_{s}^{\varepsilon}\rangle ds +2p\int_{0} ^{t}| X_{s}^{\varepsilon}|^{2p-2}\langle\sigma_1( X_{s}^{\varepsilon}), dW^{Q_1}_s\rangle\\
 \!\!\!\!\!\!\!\!&& +p\int_{0} ^{t}|X_{s}^{\varepsilon}|^{2p-2}|\sigma_1( X_{s}^{\varepsilon})|_{\mathcal{L}_{Q_1}}^2ds+2p(p-1)\int_{0}^{t}|X_{s}^{\varepsilon}|^{2p-4}|\sigma_1(X_{s}^{\varepsilon})^*X^\varepsilon_s|_{U_1}^2ds.
\end{eqnarray*}
Then by Burkholder-Davis-Gundy's inequality, \eref{F3.4} and Lemma \ref{Property B1}, it holds that
\begin{eqnarray*}
&&\mathbb{E}\left(\sup_{t\in[0, T]}|X_{t}^{\vare}|^{2p}\right)+2p\EE\left(\int^T_0|X_{t}^{\vare}|^{2p-2}\|X_{t}^{\vare}\|^2_1dt\right)\\
\leq\!\!\!\!\!\!\!\!&&|x|^{2p}+C_p+C_p\int^T_0\mathbb{E}|X_{t}^{\varepsilon}|^{2p}dt+C_p\int^T_0\mathbb{E}| Y_{t}^{\varepsilon}|^{2p}dt\\
\leq\!\!\!\!\!\!\!\!&&C_p(|x|^{2p}+|y|^{2p}+1)+C_p\int^T_0\mathbb{E}|X_{t}^{\varepsilon}|^{2p}dt+\frac{C_p}{\vare}\int^T_0\int^t_0
e^{-\frac{p\gamma}{\varepsilon}(t-s)}\left(1+\EE|X_{s}^{\varepsilon}|^{2p}\right)dsdt\\
\leq\!\!\!\!\!\!\!\!&&C_p(|x|^{2p}+|y|^{2p}+1)+C_p\int^T_0\mathbb{E}|X_{t}^{\varepsilon}|^{2p}dt.
\end{eqnarray*}
Hence, applying Gronwall's inequality implies
\begin{eqnarray*}
\mathbb{E}\left(\sup_{t\in[0, T]}|X_{t}^{\vare}|^{2p}\right)+2p\EE\left(\int^T_0|X_{t}^{\vare}|^{2p-2}\|X_{t}^{\vare}\|^2_1dt\right)
\leq\!\!\!\!\!\!\!\!&&C_{p,T}(|x|^{2p}+|y|^{2p}+1),\label{4.4.4}
\end{eqnarray*}
which also gives
\begin{eqnarray*}
\mathbb{E}|Y_{t}^{\varepsilon}|^{2p}\leq
C_{p,T}\left(1+|x|^{2p}+|y|^{2p}\right).\label{4.4.5}
\end{eqnarray*}
The proof is complete.
\end{proof}

\vspace{0.3cm}
We are going to use the approach of time discretization to prove our main result, so the H\"{o}lder continuity of time for $X_{t}^{\varepsilon}$ always plays an important role. To this purpose, the condition of the initial value $x\in H^{\theta}$ for some $\theta>0$ will be assumed, for example, see \cite[Proposition 4.4]{C1}, \cite[Lemma 3,4]{DSXZ} and \cite[Proposition 9]{GP}. However, we will prove the following lemma instead of studying the H\"{o}lder continuity of time under the assumption of the initial value $x\in H$, and it would be enough to prove our main result. Hence, the techniques used here are very helpful to weaken the regularity of initial value $x$. The main idea of its proof is inspired from \cite[Lemma 2.8]{MZ}.

\begin{lemma} \label{COX}
For any $T>0$, $\vare\in(0,1)$ and $\delta>0$ small enough, there exists a constant $C_{T}>0$ such that for any $x,y\in H$
\begin{align}
\mathbb{E}\left[\int^{T}_0|X_{t}^{\varepsilon}-X_{t(\delta)}^{\varepsilon}|^2 dt\right]\leq C_{T}\delta^{1/2}(1+|x|^3+|y|^3),\label{F3.7}
\end{align}
where $t(\delta):=[\frac{t}{\delta}]\delta$ and $[s]$ denotes the largest integer which is no more than $s$.
\end{lemma}

\begin{proof}
By \eref{F3.1}, it is easy to get that
\begin{eqnarray}
&&\mathbb{E}\left[\int^{T}_0|X_{t}^{\varepsilon}-X_{t(\delta)}^{\varepsilon}|^2dt\right]\nonumber\\
=\!\!\!\!\!\!\!\!&& \mathbb{E}\left(\int^{\delta}_0|X_{t}^{\varepsilon}-x|^2dt\right)+\mathbb{E}\left[\int^{T}_{\delta}|X_{t}^{\varepsilon}-X_{t(\delta)}^{\varepsilon}|^2dt\right]\nonumber\\
\leq\!\!\!\!\!\!\!\!&& C\delta +2\mathbb{E}\left(\int^{T}_{\delta}|X_{t}^{\varepsilon}-X_{t-\delta}^{\varepsilon}|^2dt\right)+2\mathbb{E}\left(\int^{T}_{\delta}|X_{t(\delta)}^{\varepsilon}-X_{t-\delta}^{\varepsilon}|^2dt\right).\label{F3.8}
\end{eqnarray}
Then, we estimate the second term on the right-hand side of \eref{F3.8} firstly. According to It\^{o}'s formula, we have
\begin{eqnarray}
|X_{t}^{\varepsilon}-X_{t-\delta}^{\varepsilon}|^{2}=\!\!\!\!\!\!\!\!&&2\int_{t-\delta} ^{t}\langle AX_{s}^{\varepsilon}-B(X_{s}^{\varepsilon}), X_{s}^{\varepsilon}-X_{t-\delta}^{\varepsilon}\rangle ds+ 2\int_{t-\delta} ^{t}\langle f(X_{s}^{\varepsilon}, X_{s}^{\varepsilon}), X_{s}^{\varepsilon}-X_{t-\delta}^{\varepsilon}\rangle ds\nonumber \\
 \!\!\!\!\!\!\!\!&& +\int_{t-\delta} ^{t}|\sigma_1(X_{s}^{\varepsilon})|_{\mathcal{L}_{Q_1}}^2ds
 +2\int_{t-\delta} ^{t}\langle X_{s}^{\varepsilon}-X_{t-\delta},  \sigma_1(X_{s}^{\varepsilon})dW^{Q_1}_s\rangle \nonumber\\
:=\!\!\!\!\!\!\!\!&&I_{1}(t)+I_{2}(t)+I_{3}(t)+I_{4}(t).  \label{F3.9}
\end{eqnarray}
For the first term $I_1(t)$, by H\"{o}lder's inequality and Corollary \ref{Property B2}, there exists a constant $C>0$ such that
\begin{eqnarray}  \label{REGX1}
&&\mathbb{E}\left(\int^{T}_{\delta}|I_{1}(t)|dt\right)\nonumber\\
\leq\!\!\!\!\!\!\!\!&& C\mathbb{E}\left(\int^{T}_{\delta}\int_{t-\delta} ^{t}\| AX_{s}^{\varepsilon}-B(X_{s}^{\varepsilon})\|_{-1}
\|X_{s}^{\varepsilon}-X_{t-\delta}^{\varepsilon}\|_{1} ds dt\right)\nonumber\\
\leq\!\!\!\!\!\!\!\!&&C\left[\mathbb{E}\int^{T}_{\delta}\int_{t-\delta} ^{t}\|AX_{s}^{\varepsilon}-B(X_{s}^{\varepsilon})\|^2_{-1}dsdt\right]^{1/2}
\left[\mathbb{E}\int^{T}_{\delta}\int_{t-\delta} ^{t}\|X_{s}^{\varepsilon}-X_{t-\delta}^{\varepsilon}\|^2_{1} dsdt\right]^{1/2}\nonumber\\
\leq\!\!\!\!\!\!\!\!&&C\left[\delta\mathbb{E}\int^{T}_0\|X_{s}^{\varepsilon}\|^2_{1}(1+|X_{s}^{\varepsilon}|^2)ds\right]^{1/2}\cdot\left[\delta\mathbb{E}\int^{T}_0\|X_{s}^{\varepsilon}\|^2_{1}ds\right]^{1/2}\nonumber\\
\leq\!\!\!\!\!\!\!\!&&C_{T}\delta(1+|x|^3+|y|^3),
\end{eqnarray}
where we use the Fubini theorem and \eref{F3.1} in the third and fourth inequalities respectively.

For $I_{2}(t)$ and $I_3(t)$, by condition \ref{A1} and \eref{F3.1}, we get
\begin{eqnarray}\label{REGX2}
&&\mathbb{E}\left(\int^{T}_{\delta}|I_{2}(t)|dt\right)\nonumber\\
\leq\!\!\!\!\!\!\!\!&&C\mathbb{E}\left(\int^{T}_{\delta}\int_{t-\delta} ^{t}(1+|X_{s}^{\varepsilon}|+|Y_{s}^{\varepsilon}|)(|X_{s}^{\varepsilon}|+|X_{t-\delta}^{\varepsilon}|)ds dt\right)\nonumber\\
\leq\!\!\!\!\!\!\!\!&&C\delta\mathbb{E}\left[\sup_{s\in[0,T]}(1+|X_{s}^{\varepsilon}|^2)\right]+C\mathbb{E}\left[\sup_{s\in[0,T]}|X_{s}^{\varepsilon}|\int^T_{\delta}\int^t_{t-\delta}|Y^{\vare}_s|dsdt\right]\nonumber\\
\leq\!\!\!\!\!\!\!\!&&C\delta\mathbb{E}\left[\sup_{s\in[0,T]}(1+|X_{s}^{\varepsilon}|^2)\right]\!\!+\!\!C_T\delta^{1/2}\mathbb{E}\left[\sup_{s\in[0,T]}|X_{s}^{\varepsilon}|^2\right]^{1/2}\!\!\!\!\!\mathbb{E}\left(\int^T_{\delta}\int^t_{t-\delta}|Y_{s}^{\varepsilon}|^2dsdt\right)^{1/2}\nonumber\\
\leq\!\!\!\!\!\!\!\!&&C_{T}\delta(1+|x|^2+|y|^2)
\end{eqnarray}
and
\begin{eqnarray}\label{REGX2a}
\mathbb{E}\left(\int^{T}_{\delta}|I_{3}(t)|dt\right)\leq\!\!\!\!\!\!\!\!&&C\mathbb{E}\left(\int^{T}_{\delta}\int_{t-\delta} ^{t}(1+|X_{s}^{\varepsilon}|^2)ds dt\right)\nonumber\\
\leq\!\!\!\!\!\!\!\!&&C_T\delta\mathbb{E}\left[\sup_{s\in[0,T]}(1+|X_{s}^{\varepsilon}|^2)\right]\nonumber\\
\leq\!\!\!\!\!\!\!\!&&C_{T}\delta(1+|x|^2+|y|^2).
\end{eqnarray}

For $I_{4}(t)$, applying Burkholder-Davies-Gundy's inequality implies
\begin{eqnarray}  \label{REGX3}
\mathbb{E}\left(\int^{T}_{\delta}|I_{4}(t)|dt\right)\leq\!\!\!\!\!\!\!\!&&C\mathbb{E}\int^{T}_{\delta}\left[\int_{t-\delta} ^{t}|\sigma_1(X_{s}^{\varepsilon})|^2_{\mathcal{L}_{Q_1}}|X_{s}^{\varepsilon}-X_{t-\delta}^{\varepsilon}|^2 ds\right]^{1/2}dt\nonumber\\
\leq\!\!\!\!\!\!\!\!&&C_T\left[\mathbb{E}\int^{T}_{\delta}\int_{t-\delta} ^{t}(1+|X_{s}^{\varepsilon}|^2)|X_{s}^{\varepsilon}-X_{t-\delta}^{\varepsilon}|^2dsdt\right]^{1/2}\nonumber\\
\leq\!\!\!\!\!\!\!\!&&C_{T}\delta^{1/2}\left[\mathbb{E}\sup_{s\in[0,T]}\left(1+|X_{s}^{\varepsilon}|^4\right)\right]^{1/2}\nonumber\\
\leq\!\!\!\!\!\!\!\!&&C_{T}\delta^{1/2}(1+|x|^2+|y|^2).
\end{eqnarray}
Combining estimates \eref{F3.9}-\eref{REGX3} together, we can deduce that
\begin{eqnarray}
\mathbb{E}\left(\int^{T}_{\delta}|X_{t}^{\varepsilon}-X_{t-\delta}^{\varepsilon}|^2dt\right)\leq\!\!\!\!\!\!\!\!&&C_{T}\delta^{1/2}(1+|x|^3+|y|^3). \label{F3.13}
\end{eqnarray}
By the similar argument above, we can also get
\begin{eqnarray}
\mathbb{E}\left(\int^{T}_{\delta}|X_{t(\delta)}^{\varepsilon}-X_{t-\delta}^{\varepsilon}|^2dt\right)\leq\!\!\!\!\!\!\!\!&&C_{T}\delta^{1/2}(1+|x|^3+|y|^3). \label{F3.14}
\end{eqnarray}
Hence, \eref{F3.8}, \eref{F3.13} and \eref{F3.14} implies \eref{F3.7} holds. The proof is complete.
\end{proof}


\vskip 0.3cm

\subsection{ Estimates of auxiliary process $(\hat{X}_{t}^{\varepsilon},\hat{Y}_{t}^{\varepsilon})$}

Following the idea inspired by Khasminskii in \cite{K1}, we introduce an auxiliary process $(\hat{X}_{t}^{\varepsilon},\hat{Y}_{t}^{\varepsilon})\in{H}\times H$ and divide $[0,T]$ into intervals of size $\delta$, where $\delta$ is a fixed positive number depends on $\vare$ and will be chosen later. We then construct a process $\hat{Y}_{t}^{\varepsilon}$, with initial value $\hat{Y}_{0}^{\varepsilon}=Y^{\varepsilon}_{0}=y$, and for any $k\in \mathbb{N}$ and $t\in[k\delta,\min((k+1)\delta,T)]$,
\begin{eqnarray}
\hat{Y}_{t}^{\varepsilon}=\hat{Y}_{k\delta}^{\varepsilon}+\frac{1}{\varepsilon}\int_{k\delta}^{t}A\hat{Y}_{s}^{\varepsilon}ds+\frac{1}{\varepsilon}\int_{k\delta}^{t}
g(X_{k\delta}^{\varepsilon},\hat{Y}_{s}^{\varepsilon})ds+\frac{1}{\sqrt{\varepsilon}}\int_{k\delta}^{t}\sigma_2(X_{k\delta}^{\varepsilon},\hat{Y}_{s}^{\varepsilon})dW^{Q_{2}}_s,\label{4.6a}
\end{eqnarray}
which is equivalent to
$$
d\hat{Y}_{t}^{\vare}=\frac{1}{\vare}\left[A\hat{Y}_{t}^{\vare}+g\left(X^{\vare}_{t(\delta)},\hat{Y}_{t}^{\vare}\right)\right]dt+\frac{1}{\sqrt{\vare}}\sigma_2\left(X^{\vare}_{t(\delta)},\hat{Y}_{t}^{\vare}\right)dW^{Q_2}_t,\quad \hat{Y}_{0}^{\vare}=y.
$$
Also, we define the process $\hat{X}_{t}^{\varepsilon}$ by integral
\begin{align}
\hat{X}_{t}^{\varepsilon}=x+\int_{0}^{t}A\hat{X}_{s}^{\varepsilon}ds-\int_{0}^{t}B(\hat{X}_{s}^{\varepsilon})ds+\int_{0}^{t}
f(X_{s(\delta)}^{\varepsilon},\hat{Y}_{s}^{\varepsilon})ds+\int_{0}^{t}\sigma_1(\hat{X}_{s}^{\varepsilon})dW^{Q_{1}}_s,\label{4.6b}
\end{align}
for $t\in[0,T]$. We remark that on each interval
the fast component $\hat{Y}_{t}^{\varepsilon}$ does not depend on
the slow component $\hat{X}_{t}^{\varepsilon}$, but only on the
value of $X_{t}^{\vare}$ at the first point of the interval.

\vspace{0.2cm}
By the construction of $(\hat{X}_{t}^{\varepsilon},
\hat{Y}_{t}^{\varepsilon})$, we can obtain the following
estimates which will be used below. Because the proof almost follows
the same steps in Lemma \ref{PMY}, we omit the proof here.

\begin{lemma} \label{MDYa}
For any $x, y\in H$, $T>0$ and $\vare\in(0,1)$, there exists a constant
$C_{T}>0$ such that
\begin{eqnarray}
\sup_{t\in[0,T]}\mathbb{E}|\hat{Y}_{t}^{\vare}|^2\leq
C_{T}(1+|x|^2+|y|^2) \label{3.13a}
\end{eqnarray}
and
\begin{eqnarray}
\mathbb{E}\left(\sup_{t\in[0,T]}|\hat{X}_{t}^{\varepsilon}|^2\right)+\EE\left(\int_0^T\|\hat{X}_{t}^{\varepsilon}\|_1^2dt\right)\leq C_T(|x|^2+|y|^2+1).\label{4.5a}
\end{eqnarray}
\end{lemma}
\vskip 0.3cm
Now, we will establish the difference between $Y_{t}^{\varepsilon}$ and $\hat{Y}_{t}^{\varepsilon}$, and furthermore the difference $X^{\varepsilon}_t$ and $\hat{X}_{t}^{\varepsilon}.$
\begin{lemma} \label{DEY}
For any $x, y\in H$, $T>0$ and $\vare\in(0,1)$, there exists a constant $C_{R,T}>0$ such that
\begin{eqnarray}
\mathbb{E}\left(\int_0^{T}|Y_{t}^{\varepsilon}-\hat{Y}_{t}^{\varepsilon}|^2dt\right)\leq C_{T}\left(1+|x|^3+|y|^3\right)\delta^{1/2}. \label{3.14}
\end{eqnarray}
\end{lemma}

\begin{proof}
Let $\rho^{\vare}_t:=Y_{t}^{\varepsilon}-\hat{Y}_{t}^{\varepsilon}$
Then, it is easy to see that $\rho^{\vare}_t$ satisfies the following equation:
\begin{equation}\left\{\begin{array}{l}
\displaystyle d\rho^{\vare}_t=\frac{1}{\varepsilon}\left[A\rho^{\vare}_t+g(X_t^\varepsilon,Y_t^\varepsilon)-g(X_{t(\delta)}^\varepsilon,\hat{Y}_t^\varepsilon)\right]dt+\frac{1}{\sqrt{\vare}}\left[\sigma_2(X_t^\varepsilon,Y_t^\varepsilon)-\sigma_2
(X_{t(\delta)}^\varepsilon,\hat{Y}_t^\varepsilon)\right]dW_t^{Q_2},\\
\rho^{\vare}_0=0,\end{array}\right. \nonumber
\end{equation}
Thus, applying It\^{o}'s formula and taking expectation, we have
\begin{eqnarray*}
\EE|\rho^{\vare}_t|^2
=\!\!\!\!\!\!\!\!&&-\frac{2}{\varepsilon}\int^t_0\EE\|\rho^{\vare}_s\|_1^2ds+\frac{2}{\varepsilon}\int^t_0\EE\langle g(X_s^\varepsilon,Y_s^\varepsilon)-g(X_{s(\delta)}^\varepsilon,\hat{Y}_s^\varepsilon),\rho^{\vare}_s\rangle ds \nonumber\\
\!\!\!\!\!\!\!\!&&+\frac{1}{\varepsilon}\int^t_0\EE\|\sigma_2(X_s^\varepsilon,Y_s^\varepsilon)-\sigma_2
(X_{s(\delta)}^\varepsilon,\hat{Y}_s^\varepsilon)\|^2ds.
\end{eqnarray*}
Then by condition \ref{A4}, there exits $\beta>0$ such that
\begin{eqnarray*}
\frac{d}{dt}\EE|\rho^{\vare}_t|^2
=\!\!\!\!\!\!\!\!&&-\frac{2}{\varepsilon}\EE\|\rho^{\vare}_t\|_1^2+\frac{2}{\varepsilon}\EE\langle g(X_t^\varepsilon,Y_t^\varepsilon)-g(X_{t(\delta)}^\varepsilon,\hat{Y}_t^\varepsilon),\rho^{\vare}_t\rangle \nonumber\\
\!\!\!\!\!\!\!\!&&+\frac{1}{\varepsilon}\EE\|\sigma_2(X_t^\varepsilon,Y_t^\varepsilon)-\sigma_2
(X_{t(\delta)}^\varepsilon,\hat{Y}_t^\varepsilon)\|^2\nonumber\\
\leq\!\!\!\!\!\!\!\!&&-\frac{2\lambda_1}{\varepsilon}\EE|\rho^{\vare}_t|^2+\frac{C}{\varepsilon}\EE\left(|X_t^\varepsilon-X_{t(\delta)}^\varepsilon|\cdot|\rho^{\vare}_t|\right)+\frac{2L_g}{\varepsilon}\EE|\rho^{\vare}_t|^2  \nonumber\\
\!\!\!\!\!\!\!\!&&+\frac{1}{\varepsilon}\EE\left(C|X_t^\varepsilon-X_{t(\delta)}^\varepsilon|+L_{\sigma_2}|\rho^{\vare}_t|\right)^2\nonumber\\
\leq\!\!\!\!\!\!\!\!&&-\frac{\beta}{\varepsilon}\EE|\rho^{\vare}_t|^2+\frac{C}{\varepsilon}\EE|X_t^\varepsilon-X_{t(\delta)}^\varepsilon|^2.\nonumber
\end{eqnarray*}
Therefore, comparison theorem yields that
\begin{eqnarray*}
\EE|\rho^{\vare}_t|^2\leq\frac{C}{\varepsilon}\int_0^te^{-\frac{\beta(t-s)}{\vare}}\EE|X_s^\varepsilon-X_{s(\delta)}^\varepsilon|^2ds.
\end{eqnarray*}
Then by Fubini's theorem, for any $T>0$,
\begin{eqnarray*}
\EE\left(\int_0^T|\rho^{\vare}_t|^2dt\right)\leq\!\!\!\!\!\!\!\!&& \frac{C}{\varepsilon}\int_0^T\int^t_0e^{-\frac{\beta(t-s)}{\vare}}\EE|X_s^\varepsilon-X_{s(\delta)}^\varepsilon|^2dsdt\nonumber\\
=\!\!\!\!\!\!\!\!&&  \frac{C}{\varepsilon}\EE\left[\int_0^T|X_s^\varepsilon-X_{s(\delta)}^\varepsilon|^2\left(\int^T_s e^{-\frac{\beta(t-s)}{\vare}}dt\right)ds\right]\nonumber\\
\leq\!\!\!\!\!\!\!\!&& C\EE\left(\int_0^T|X_s^\varepsilon-X_{s(\delta)}^\varepsilon|^2 ds\right).
\end{eqnarray*}
By Lemma \ref{COX}, we obtain
\begin{eqnarray*}
\mathbb{E}\left(\int_0^{T}|\rho^{\vare}_t|^2dt\right)\leq C_{T}(|x|^3+|y|^3+1)\delta^{1/2}.
\end{eqnarray*}
The proof is complete.
\end{proof}

In order to estimate the difference process $|X_{t}^{\vare}-\hat{X}_{t}^{\vare}|$,  we need to construct the following stopping time, i.e., for fixed $\vare\in(0,1)$, $R>0$,
$$
\tau_R^\varepsilon:=\inf\left\{t\geq0:\int_0^t\|X_s^\varepsilon\|_1^2ds\geq R\right\}.
$$
Then we will control  $|X_{t}^{\vare}-\hat{X}_{t}^{\vare}|$ when $t$ is before this stopping time.
\begin{lemma} \label{DEX}
For any $x, y\in H$, $T>0$ and $\vare\in(0,1)$, the following fact holds.
\begin{align*}
\mathbb{E}\Big(\sup_{t\in[0, T\wedge\tau_R^\varepsilon]}|X_{t}^{\vare}-\hat{X}_{t}^{\vare}|^2\Big)\leq C_{R,T}\left(1+|x|^3+|y|^3\right)\delta^{1/2}.
\end{align*}
\end{lemma}

\begin{proof}
Put $Z_t^\varepsilon:=X_{t}^{\vare}-\hat{X}_{t}^{\vare}$, then we have $Z^{\vare}_0=0$ and
\begin{eqnarray*}
dZ_t^\varepsilon=\!\!\!\!\!\!\!\!&&AZ_t^\varepsilon dt-\left[B(X_t^\varepsilon)-B(\hat{X}_t^\varepsilon)\right]dt
+\left[f(X_t^\varepsilon,Y_t^\varepsilon)-f(X_{t(\delta)}^\varepsilon,\hat{Y}_t^\varepsilon)\right]dt
\\\!\!\!\!\!\!\!\!&&+\left[\sigma_1(X_t^\varepsilon)-\sigma_1(\hat{X}_t^\varepsilon)\right]dW_t^{Q_1}.
\end{eqnarray*}
By It\^{o}'s formula and Corollary \ref{Property B2}, we have
\begin{eqnarray}
|Z_t^\varepsilon|^2
=\!\!\!\!\!\!\!\!&&\int_0^t-2\|Z_s^\varepsilon\|_1^2ds-2\int_0^t\langle B(X_s^\varepsilon)-B(\hat{X}_s^\varepsilon), Z_s^\varepsilon\rangle ds\nonumber\\
\!\!\!\!\!\!\!\!&&+2\int_0^t\left\langle\left[f(X_s^\varepsilon,Y_s^\varepsilon)-f(X_{s(\delta)}^\varepsilon,\hat{Y}_s^\varepsilon)\right],Z_s^\varepsilon\right\rangle ds\nonumber\\
\!\!\!\!\!\!\!\!&&+2\int_0^t\left\langle Z_s^\varepsilon,[\sigma_1(X_s^\varepsilon)-\sigma_1(\hat{X}_s^\varepsilon)]dW_s^{Q_1}\right\rangle+\int_0^t|\sigma_1(X_s^\varepsilon)
-\sigma_1(\hat{X}_s^\varepsilon)|_{\mathcal{L}_{Q_1}}^2ds\nonumber\\
\leq\!\!\!\!\!\!\!\!&&-\int_0^t\|Z_s^\varepsilon\|_1^2ds+C\int_0^t|Z_s^\varepsilon|^2\|X_s^\varepsilon\|_1^2ds+C\int_0^t|X_s^\varepsilon-X_{s(\delta)}^\varepsilon|^2ds\nonumber\\
\!\!\!\!\!\!\!\!&&+C\int_0^t|Y_s^\varepsilon-\hat{Y}_s^\varepsilon|^2ds+C\int_0^t|Z_s^\varepsilon|^2ds+2\int_0^t\langle Z_s^\varepsilon,[\sigma_1(X_s^\varepsilon)-\sigma_1(\hat{X}_s^\varepsilon)]dW_s^{Q_1}\rangle,\nonumber
\end{eqnarray}
which yields
\begin{eqnarray*}
\sup_{t\in[0, T\wedge\tau_R^\varepsilon]}|Z_t^\varepsilon|^2
\leq\!\!\!\!\!\!\!\!&&C\int_0^{T\wedge\tau_R^\varepsilon}(1+\|X_s^\varepsilon\|_1^2)|Z_s^\varepsilon|^2ds
+C\int_0^{T\wedge\tau_R^\varepsilon}|X_s^\varepsilon-X_{s(\delta)}^\varepsilon|^2ds
\\\!\!\!\!\!\!\!\!&&+C\int_0^{T\wedge\tau_R^\varepsilon}|Y_s^\varepsilon-\hat{Y}_s^\varepsilon|^2ds
+2\sup_{t\in[0, T\wedge\tau_R^\varepsilon]}\left|\int_0^t\langle Z_s^\varepsilon,[\sigma_1(X_s^\varepsilon)-\sigma_1(\hat{X}_s^\varepsilon)]dW_s^{Q_1}\rangle\right|.
\end{eqnarray*}
Using Gronwall's inequality and the definition of $\tau^{\vare}_{R}$, we can deduce
\begin{eqnarray}
\sup_{t\in[0, T\wedge\tau_R^\varepsilon]}|Z_s^\varepsilon|^2
\leq\!\!\!\!\!\!\!\!&&C_{R,T}\Big[\int_0^{T\wedge\tau_R^\varepsilon}|X_s^\varepsilon-X_{s(\delta)}^\varepsilon|^2ds+\int_0^{T\wedge\tau_R^\varepsilon}|Y_s^\varepsilon-\hat{Y}_s^\varepsilon|^2ds\nonumber\\
\!\!\!\!\!\!\!\!&&+\sup_{t\in[0, T\wedge\tau_R^\varepsilon]}\left|\int_0^t\langle Z_s^\varepsilon,[\sigma_1(X_s^\varepsilon)-\sigma_1(\hat{X}_s^\varepsilon)]dW_s^{Q_1}\rangle\right|\Big].\nonumber
\end{eqnarray}
Then  by Burkholder-Davis-Gundy's inequality, we obtain
\begin{eqnarray*}
\mathbb{E}\left(\sup_{t\in[0, T\wedge\tau_R^\varepsilon]}|Z_t^\varepsilon|^2\right)
\leq\!\!\!\!\!\!\!\!&& C_{R,T}\left(1+|x|^3+|y|^3\right)\delta^{1/2}+\frac{1}{2}\mathbb{E}
\left(\sup_{t\in[0, T\wedge\tau_R^\varepsilon]}|Z_t^\varepsilon|^2\right)
\\\!\!\!\!\!\!\!\!&&+C_{R,T}\mathbb{E}\left(\int_0^{T\wedge\tau_R^\varepsilon}|Z_t^\varepsilon|^2dt\right),  \nonumber
\end{eqnarray*}
which implies
\begin{eqnarray*}
\mathbb{E}\left(\sup_{t\in[0, T\wedge\tau_R^\varepsilon]}|Z_t^\varepsilon|^2\right)\leq C_{R,T}\left(1+|x|^3+|y|^3\right)\delta^{1/2}+C_{R,T}\int_0^T\mathbb{E}\left(\sup_{s\in[0, t\wedge\tau_R^\varepsilon]}|Z_s^\varepsilon|^2\right)dt.
\end{eqnarray*}
By Gronwall's inequality again, we get
$$\mathbb{E}\left(\sup_{t\in[0, T\wedge\tau_R^\varepsilon]}|Z_t^\varepsilon|^2\right)\leq C_{R,T}\left(1+|x|^3+|y|^3\right)\delta^{1/2}.$$
The proof is complete.
\end{proof}

\subsection{Frozen and averaged equation}
We first recall the frozen equation associate to fast motion for fixed slow component $x\in {H}$, i.e.,
\begin{eqnarray}
\left\{ \begin{aligned}
&dY_{t}=[AY_{t}+g(x,Y_{t})]dt+\sigma_2(x,Y_t)d\bar{W}_{t}^{Q_{2}},\\
&Y_{0}=y,
\end{aligned} \right.\label{FEQ1}
\end{eqnarray}
where $\bar{W}_{t}^{Q_{2}}$ is $Q_2$-Wiener process independent of $W^{Q_1}_t$ and $W^{Q_2}_t$. Notice that $g(x,\cdot)$ and $\sigma_2(x,\cdot)$ is Lipshcitz continuous, it is easy to prove for any fixed $x\in {H}$ and any initial data $y\in H$, equation $\eref{FEQ1}$ has a unique mild solution $Y_{t}^{x,y}$. Let $P^{x}_t$ be the transition semigroup of $Y_{t}^{x,y}$,
that is, for any bounded measurable function $\varphi$ on $H$,
\begin{eqnarray*}
P^x_t \varphi(y)= \mathbb{E} \left[\varphi\left(Y_{t}^{x,y}\right)\right], \quad y \in H,\ \ t>0.
\end{eqnarray*}
Similar as the argument in \cite[Section 2.1]{C1}, under the condition \ref{A4}, it is easy to prove that $\EE|Y_{t}^{x,y}|^2\leq C(1+|x|^2+e^{-\delta_1 t}|y|^2)$ for some $\delta_1>0$, and $P^x_t$ has unique invariant measure $\mu^x$. We here give the following asymptotic behavior of $P^x_t$ proved in \cite[ (2.13)]{C1}.

\begin{theorem}\label{Rem 4.1}
For any given value $x,y\in H$, there exists a unique invariant measure $\mu^x$ for $\eref{FEQ1}$. Moreover, there exist $C>0$ and $\eta>0$ such that for any Lipschitz function $\varphi:H\rightarrow \mathbb{R}$,
\begin{equation}
\Big|P^x_t\varphi(y)-\int_{H}\varphi(z)\mu^x(dz)\Big|\leq C(1+|x|+|y|)e^{-\frac{\eta t}{2}}|\varphi|_{Lip},
\end{equation}
where $|\varphi|_{Lip}=\sup_{x,y\in H}\frac{|\varphi(x)-\varphi(y)|}{|x-y|}$.
\end{theorem}


 \vskip 0.3cm

Next, we recall the corresponding averaged equation, i.e.,
\begin{equation}\left\{\begin{array}{l}
\displaystyle d\bar{X}_{t}=A\bar{X}_{t}dt-B(\bar{X}_{t})dt+\bar{f}(\bar{X}_{t})dt+\sigma_1(\bar{X}_{t})dW^{Q_{1}}_t,\\
\bar{X}_{0}=x,\end{array}\right. \label{3.1}
\end{equation}
with the average
\begin{align*}
\bar{f}(x)=\int_{H}f(x,y)\mu^{x}(dy),\quad x\in {H},
\end{align*}
where $\mu^{x}$ is the unique invariant measure for equation $\eref{FEQ1}$.
\vskip 0.3cm

Due to the Lipschitz of $f$, by a standard method, it is easy to check $\bar{f}$ is Lipschitz and equation \eref{3.1} has a unique solution. Then similar with the argument in Lemma \ref{PMY}, we also have the following estimate.
\begin{lemma}\label{L3.8}
For any $T>0$, $p\geq 1$, there exists a positive constant $C_{p,T}$ such that for any $x\in H$
\begin{align*}
\mathbb{E}\left(\sup_{t\in[0,T]}|\bar{X}_{t}|^{2p}\right)+\EE\left(\int_0^T|\bar{X}_{t}|^{2p-2}\|\bar{X}_{t}\|_1^2dt\right)\leq C_{p,T}(1+|x|^{2p}).
\end{align*}
\end{lemma}

\vskip 0.3cm
 In the next lemma, we shall deal with the difference process $\hat{X}_{t}^{\vare}-\bar{X}_t$. To this end, we shall construct another stopping time, i.e., for fixed $\vare\in(0,1)$, $R>0$,
 $$\tilde{\tau}_R^\varepsilon:=\inf\left\{t\geq 0:\int_0^t\|X^{\vare}_s\|_1^2ds+\int_0^t\|\hat{X}_s^\varepsilon\|_1^2ds+\int_0^t\|\bar{X}_s\|_1^2ds+\int_0^t\|J_s^\varepsilon\|_\frac{1}{2}^2ds\geq R\right\},$$
 where $J_t^\varepsilon$ is the solution of the following equation:
\begin{equation}\left\{\begin{array}{l}
\displaystyle dJ_t^\varepsilon=A J_t^\varepsilon dt+\left[f(X_{t(\delta)}^\varepsilon,\hat{Y}_t^\varepsilon)-\bar{f}(X_{t(\delta)}^\varepsilon)\right]dt,\\
J_0^\varepsilon=0,\end{array}\right. \label{E3.20}
\end{equation}
which satisfies
\begin{equation*}
J_t^\varepsilon=\int_0^te^{(t-s)A}\left[f(X_{s(\delta)}^\varepsilon,\hat{Y}_s^\varepsilon)-\bar{f}(X_{s(\delta)}^\varepsilon)\right]ds.
\end{equation*}
We remark here that the reason why we introduce $J_t^\varepsilon$ in $\tilde{\tau}_R^\varepsilon$ is a technical treatment in the proof of following Lemma.
\begin{lemma} \label{ESX}
For any $x,y\in H$, $T>0$ and $\vare\in(0,1)$, then there exists a constant $C_{R,T}>0$ such that
\begin{align*}
\mathbb{E}\left(\sup_{t\in[0, T\wedge{\tilde{\tau}_R^\varepsilon}]}|\hat{X}_{t}^{\vare}-\bar{X}_{t}|^2\right)\leq C_{R,T}(1+|x|^3+|y|^{3})\left(\frac{\varepsilon^\frac{1}{2}}{\delta}+\delta^{1/2}\right).
\end{align*}
\end{lemma}

\begin{proof}
 The proof is divided into two steps.

\vspace{2mm}
\textbf{Step 1. (Splitting $\hat{X}_t^\varepsilon-\bar{X}_t$ into two terms)}:
Let $L_t^\varepsilon:=\hat{X}_t^\varepsilon-\bar{X}_t$ and $V_t^\varepsilon:=L_t^\varepsilon-J_t^\varepsilon$. From equations \eref{4.6b}, \eref{3.1} and \eref{E3.20}, we can write
\begin{eqnarray*}
dV_t^\varepsilon
=\!\!\!\!\!\!\!\!&&AV_t^\varepsilon dt-[B(\hat{X}_t^\varepsilon)-B(\bar{X}_t)]dt \nonumber\\
\!\!\!\!\!\!\!\!&&+[\bar{f}(X_{t(\delta)}^\varepsilon)-\bar{f}(\bar{X}_t)]dt   +[\sigma_1(\hat{X}_t^\varepsilon)-\sigma_1(\bar{X}_t)]dW_t^{Q_1}.\\
=\!\!\!\!\!\!\!\!&&AV_t^\varepsilon dt-[B(\hat{X}_t^\varepsilon)-B(\bar{X}_t)]dt+[\bar{f}(X_{t(\delta)}^\varepsilon)-\bar{f}(X_t^\varepsilon)]dt \nonumber\\\!\!\!\!\!\!\!\!&&+[\bar{f}(X_t^\varepsilon)-\bar{f}(\hat{X}_t^\varepsilon)+\bar{f}(\hat{X}_t^\varepsilon)-\bar{f}(\bar{X}_t)]dt   +[\sigma_1(\hat{X}_t^\varepsilon)-\sigma_1(\bar{X}_t)]dW_t^{Q_1}.   \nonumber
\end{eqnarray*}
Applying It\^{o}'s formula, we have
\begin{eqnarray}
|V_t^\varepsilon|^2
=\!\!\!\!\!\!\!\!&&-2\int_0^t\|V_t^\varepsilon\|^2_1 ds-2\int_0^t\langle B(\hat{X}_s^\varepsilon)-B(\bar{X}_s), V_s^\varepsilon\rangle ds \nonumber\\
\!\!\!\!\!\!\!\!&&+2\int_0^t\langle\bar{f}(X_{s(\delta)}^\varepsilon)-\bar{f}(X_s^\varepsilon)+\bar{f}(X_s^\varepsilon)-\bar{f}(\hat{X}_s^\varepsilon)
+\bar{f}(\hat{X}_s^\varepsilon)-\bar{f}(\bar{X}_s), V_s^\varepsilon\rangle \nonumber\\
\!\!\!\!\!\!\!\!&&+\int_0^t|\sigma_1(\hat{X}_s^\varepsilon)-\sigma_1(\bar{X}_s)|_{\mathcal{L}_{Q_1}}^2ds +2\int_0^t\langle V_s^\varepsilon,[\sigma_1(\hat{X}_s^\varepsilon)-\sigma_1(\bar{X}_s)]dW_s^{Q_1}\rangle \nonumber\\
:=\!\!\!\!\!\!\!\!&&-2\int_0^t\|V_s^\varepsilon\|_1^2ds+\mathcal{I}_1(t)
+\mathcal{I}_2(t)+\mathcal{I}_3(t)+\mathcal{I}_4(t).  \label{3.18}
\end{eqnarray}
For $\mathcal{I}_1(t)$, we can use Corollary \ref{Property B2} and Young's inequality to get that
\begin{eqnarray}
\mathcal{I}_1(t)
=\!\!\!\!\!\!\!\!&&-2\int_0^t\langle B(\hat{X}_s^\varepsilon)-B(\bar{X}_s+J_s^\varepsilon), V_s^\varepsilon\rangle ds-2\int_0^t\langle B(\bar{X}_s+J_s^\varepsilon)-B(\bar{X}_s), V_s^\varepsilon\rangle ds \nonumber\\
\leq\!\!\!\!\!\!\!\!&& C\int_0^t|V_s^\varepsilon|\|V_s^\varepsilon\|_1\|\hat{X}_s^\varepsilon\|_1ds+2\int_0^t\|B(\bar{X}_s+J_s^\varepsilon)
-B(\bar{X}_s)\|_{-1}\|V_s^\varepsilon\|_1ds \nonumber\\
\leq\!\!\!\!\!\!\!\!&&\int_0^t\|V_s^\varepsilon\|_1^2ds +C\int_0^t|V_s^\varepsilon|^2\|\hat{X}_s^\varepsilon\|_1^2ds
\nonumber\\
\!\!\!\!\!\!\!\!&&+C\int_0^t\|J_s^\varepsilon\|_{1/2}^2\left(\|\bar{X}_s
+J_s^\varepsilon\|_{1/2}^2+\|\bar{X}_s\|^2_{1/2}\right)ds. \label{3.19}
\end{eqnarray}
For $\mathcal{I}_2(t)$,  we have by using the Lipschitz continuous property of $\bar{f}$ and Young's inequality that
\begin{eqnarray}
\mathcal{I}_2(t)\leq C\int_0^t(|X_{s(\delta)}^\varepsilon-X_s^\varepsilon|^2+|X_s^\varepsilon-\hat{X}_s^\varepsilon|^2+|L_s^\varepsilon|^2)ds+C\int_0^t|V_s^\varepsilon|^2ds. \label{3.20}
\end{eqnarray}
According to condition \ref{A1}, it is easy to see that
\begin{eqnarray}
\mathcal{I}_3(t)\leq C\int_0^t|L_s^\varepsilon|^2ds. \label{3.20a}
\end{eqnarray}
Combining estimates \eref{3.18}-\eref{3.20a} together, we obtain
\begin{eqnarray*}
|V_t^\varepsilon|^2+\int_0^t\|V_s^\varepsilon\|_1^2ds
&&\!\!\!\!\!\!\!\!\leq C\int_0^t(1+\|\hat{X}_s^\varepsilon\|_1^2)|V_s^\varepsilon|^2ds
+C\int_0^t\|J_s^\varepsilon\|_{1/2}^2(\|\bar{X}_s
+J_s^\varepsilon\|_{1/2}^2+\|\bar{X}_s\|^2_{1/2})ds\\
\!\!\!\!\!\!\!\!&&+C\int_0^t(|X_{s(\delta)}^\varepsilon-X_s^\varepsilon|^2+
|X_s^\varepsilon-\hat{X}_s^\varepsilon|^2+|L_s^\varepsilon|^2)ds
+\mathcal{I}_4(t).
\end{eqnarray*}
As a result, it follows from Gronwall's inequality that
\begin{eqnarray*}
|V_t^\varepsilon|^2+\int_0^t\|V_t^\varepsilon\|_1^2ds
&&\!\!\!\!\!\!\!\!\leq C\Big[\int_0^t\left(|X_{s(\delta)}^\varepsilon-X_s^\varepsilon|^2+|X_s^\varepsilon-\hat{X}_s^\varepsilon|^2+|L_s^\varepsilon|^2\right)ds\\
\!\!\!\!\!\!\!\!&&+\int_0^t\|J_s^\varepsilon\|_{1/2}^2\left(\|\bar{X}_s
+J_s^\varepsilon\|_{1/2}^2+\|\bar{X}_s\|_{1/2}^2\right)ds+\mathcal{I}_4(t)\Big]e^{\int_0^t(\|\hat{X}_s^\varepsilon\|_1^2+1)ds}.
\end{eqnarray*}
By the definition of stopping time $\tilde{\tau}_R^\varepsilon$ and Burkholder-Davis-Gundy's inequality, we can deduce that
\begin{eqnarray*}
\mathbb{E}\left(\sup_{t\in[0, T\wedge{\tilde{\tau}_R^\varepsilon}]}|V_t^\varepsilon|^2\right)&&\!\!\!\!\!\!\!\!\leq C_{R,T}\mathbb{E}\left[\sup_{t\in[0, T\wedge{\tilde{\tau}_R^\varepsilon}]}\|J_t^\varepsilon\|_{1/2}^2\int_0^{T\wedge\tilde{\tau}_R^\varepsilon}(\|J_t^\varepsilon\|_{1/2}^2+\|\bar{X}_t\|_1^2)dt\right.\\
\!\!\!\!\!\!\!\! &&\left.+\int_0^{T\wedge\tilde{\tau}_R^\varepsilon}\left(|X_{t(\delta)}^\varepsilon-X_t^\varepsilon|^2+|X_t^\varepsilon-\hat{X}_t^\varepsilon|+|L_t^\varepsilon|^2\right)dt
+\sup_{t\in[0, T\wedge{\tilde{\tau}_R^\varepsilon}]}\mathcal{I}_4(t)\right]\\
&&\!\!\!\!\!\!\!\!\leq  C_{R,T}\mathbb{E}\left(\sup_{t\in[0, T\wedge{\tilde{\tau}_R^\varepsilon}]}\|J_t^\varepsilon\|_{1/2}^2\right)
+\frac{1}{2}\mathbb{E}\left(\sup_{t\in[0, T\wedge{\tilde{\tau}_R^\varepsilon}]}|V_t^\varepsilon|^2\right)\\
\!\!\!\!\!\!\!\!&&+C_{R,T}\mathbb{E}\left[\int_0^{T\wedge\tilde{\tau}_R^\varepsilon}\left(|X_{t(\delta)}^\varepsilon-X_t^\varepsilon|^2+|X_t^\varepsilon-\hat{X}_t^\varepsilon|^2+|L_t^\varepsilon|^2\right)dt\right].
\end{eqnarray*}
Then it follows
\begin{eqnarray*}
\mathbb{E}\left(\sup_{t\in[0, T\wedge{\tilde{\tau}_R^\varepsilon}]}|V_t^\varepsilon|^2\right)
&&\!\!\!\!\!\!\!\!\leq C_{R,T}\mathbb{E}\left(\sup_{t\in[0, T\wedge{\tilde{\tau}_R^\varepsilon}]}\|J_t^\varepsilon\|_{1/2}^2\right)+C_{R,T}\mathbb{E}\left(\int_0^{T\wedge\tilde{\tau}_R^\varepsilon}|L_t^\varepsilon|^2dt\right)\\
\!\!\!\!\!\!\!\!&&+C_{R,T}\mathbb{E}\left[\int_0^{T\wedge\tilde{\tau}_R^\varepsilon}\left(|X_{t(\delta)}^\varepsilon-X_t^\varepsilon|^2+|X_t^\varepsilon-\hat{X}_t^\varepsilon|^2\right)dt\right].
\end{eqnarray*}
Notice that $\tilde{\tau}^{\vare}_R\leq\tau^{\vare}_R$, we can get by Lemmas \ref{COX} and \ref{DEX} that
\begin{eqnarray*}
\mathbb{E}\left(\sup_{t\in[0, T\wedge{\tilde{\tau}_R^\varepsilon}]}|L_t^\varepsilon|^2\right)
&&\!\!\!\!\!\!\!\!\leq2\mathbb{E}\left(\sup_{t\in[0, T\wedge{\tilde{\tau}_R^\varepsilon}]}|V_t^\varepsilon|^2\right)+2\mathbb{E}\left(\sup_{t\in[0,T]}|J_t^\varepsilon|^2\right)\nonumber \\
&&\!\!\!\!\!\!\!\!\leq C_{R,T}\mathbb{E}\left(\sup_{t\in[0,T]}\|J_t^\varepsilon\|^2_{1/2}\right)+C_{R,T}\mathbb{E}\left(\int_0^{T\wedge\tilde{\tau_R}^\varepsilon}|L_t^\varepsilon|^2dt\right)\nonumber\\
\!\!\!\!\!\!\!\!&&+C_{R,T}\mathbb{E}\left[\int_0^{T\wedge\tilde{\tau}_R^\varepsilon}\left(|X_{t(\delta)}^\varepsilon-X_t^\varepsilon|^2+|X_t^\varepsilon-\hat{X}_t^\varepsilon|^2\right)dt\right]\\
&&\!\!\!\!\!\!\!\!\leq C_{R,T}\mathbb{E}\left(\sup_{t\in[0,T]}\|J_t^\varepsilon\|^2_{1/2}\right)+C_{R,T}\mathbb{E}\left(\int_0^{T\wedge\tilde{\tau}_R^\varepsilon}|L_t^\varepsilon|^2dt\right)\nonumber\\
\!\!\!\!\!\!\!\!&&+C_{R,T}\left(1+|x|^3+|y|^3\right)\delta^{1/2}.
\end{eqnarray*}
Then by Gronwall's inequality, we obtain
\begin{eqnarray}\label{5.1}
\mathbb{E}\left(\sup_{t\in[0, T\wedge{\tilde{\tau}_R^\varepsilon}]}|L_t^\varepsilon|^2\right)\leq C_{R,T}\mathbb{E}\left(\sup_{t\in[0,T]}\|J_t^\varepsilon\|^2_{1/2}\right)+C_{R,T}\left(1+|x|^3+|y|^3\right)\delta^{1/2}.
\end{eqnarray}

Hence, the proof will be finished by the following estimates \eref{3.23}, which will be done in the next step.

\vspace{2mm}
\textbf{Step 2. (The estimate for $J_t^\varepsilon$)}:
By the discussion above, it suffices to show that the following estimate holds, i.e.,
\begin{eqnarray}\label{3.23}
\mathbb{E}\left(\sup_{t\in[0,T]}\|J_t^\varepsilon\|_{1/2}^2\right)\leq C_T(1+|x|^2+|y|^2)\frac{\varepsilon^\frac{1}{2}}{\delta}.
\end{eqnarray}
To this end, we decompose $J_t^\varepsilon$ by
\begin{eqnarray*}
J_t^\varepsilon
=\!\!\!\!\!\!\!\!&&\sum_{k=0}^{[t/\delta]-1}\int_{k\delta}^{(k+1)\delta}e^{(t-s)A}\left[f(X_{k\delta}^\varepsilon,\hat{Y}_s^\varepsilon)-\bar{f}(X_{k\delta}^\varepsilon)\right]ds
\\\!\!\!\!\!\!\!\!&&+\int_{[t/\delta]\delta}^te^{(t-s)A}\left[f\left(X_{[t/\delta]\delta}^\varepsilon,\hat{Y}_s^\varepsilon\right)-\bar{f}\left(X_{[t/\delta]\delta}^\varepsilon\right)\right]ds\nonumber\\
:=\!\!\!\!\!\!\!\!&&J_1^\varepsilon(t)+J_2^\varepsilon(t).\nonumber
\end{eqnarray*}
For $J_2^\varepsilon(t)$, by Lemmas \ref{PMY} and \ref{MDYa}, it is easy to see
\begin{eqnarray}
\mathbb{E}\left(\sup_{t\in[0,T]}\|J_2^\varepsilon(t)\|_{1/2}^2\right)
\leq\!\!\!\!\!\!\!\!&&C\mathbb{E}\left[\sup_{t\in[0,T]}\left|\int_{n_t\delta}^t(t-s)^{-\frac{1}{4}}(1+|X_{[t/\delta]\delta}^\varepsilon|+|\hat{Y}_s^\varepsilon|)ds\right|^2\right]\nonumber\\
\leq\!\!\!\!\!\!\!\!&&C\left(\int_{0}^{\delta}s^{-\frac{1}{2}}ds\right)\cdot\mathbb{E}\left[\int_0^T(1+|X_{[t/\delta]\delta}^\varepsilon|^2+|\hat{Y}_s^\varepsilon|^2)ds\right]\nonumber\\
\leq\!\!\!\!\!\!\!\!&&C_T(1+|x|^2+|y|^2)\delta^\frac{1}{2}.\nonumber
\end{eqnarray}
For $J_1^\varepsilon(t)$, by the construction of $\hat{Y}_{t}^{\vare}$,
we obtain that, for any $k\in \mathbb{N}$ and $s\in[0,\delta)$,
\begin{eqnarray}
\hat{Y}_{s+k\delta}^{\vare}
=\!\!\!\!\!\!\!\!&&\hat Y_{k\delta}^{\vare}+\frac{1}{\vare}\int_{k\delta}^{k\delta+s}A\hat{Y}_{r}^{\vare}dr
+\frac{1}{\vare}\int_{k\delta}^{k\delta+s}g\left(X_{k\delta}^{\vare},\hat{Y}_{r}^{\vare}\right)dr\nonumber\\
 \!\!\!\!\!\!\!\! &&
+\frac{1}{\sqrt{\vare}}\int_{k\delta}^{k\delta+s}\sigma_2\left(X_{k\delta}^{\vare},\hat{Y}_{r}^{\vare}\right)dW^{Q_2}_r   \nonumber\\
=\!\!\!\!\!\!\!\!&&\hat Y_{k\delta}^{\vare}+\!\!\frac{1}{\vare}\int_{0}^{s}A\hat{Y}_{r+k\delta}^{\vare}dr
+\!\!\frac{1}{\vare}\int_{0}^{s}g\left(X_{k\delta}^{\vare}, \hat{Y}_{r+k\delta}^{\vare}\right)dr\nonumber\\
\!\!\!\!\!\!\!\! &&
+\!\!\frac{1}{\sqrt{\vare}}\int_{0}^{s}\sigma_2\left(X_{k\delta}^{\vare},\hat{Y}_{r+k\delta}^{\vare}\right)d\tilde{W}^{Q_2}_r, \label{E3.26}
\end{eqnarray}
where $\tilde{W}^{Q_2}_t:=W^{Q_2}_{t+k\delta}-W^{Q_2}_{k\delta}$ is the shift version of $W^{Q_2}_t$.

We construct a process $Y^{X_{k\delta}^{\vare},\hat Y_{k\delta}^{\vare}}_t$ by means of $Y^{x,y}_t \mid_{(x,y)=\left(X_{k\delta}^{\vare},\hat Y_{k\delta}^{\vare}\right)}$, where $Y^{x,y}$ is the solution to equation \eqref{FEQ1}, i.e.,
\begin{eqnarray}
Y_{\frac{s}{\vare}}^{X_{k\delta}^{\vare},\hat Y_{k\delta}^{\vare}}=\!\!\!\!\!\!\!\!&&\hat Y_{k\delta}^{\vare}
+\int_{0}^{\frac{s}{\vare}}AY_{r}^{X_{k\delta}^{\vare},\hat Y_{k\delta}^{\vare}}dr
+\int_{0}^{\frac{s}{\vare}}g\left(X_{k\delta}^{\vare}, Y_{r}^{X_{k\delta}^{\vare},\hat Y_{k\delta}^{\vare}}\right)dr\nonumber\\
\!\!\!\!\!\!\!\!&&+\int_{0}^{\frac{s}{\vare}}\sigma_2\left(X_{k\delta}^{\vare}, Y_{r}^{X_{k\delta}^{\vare},\hat Y_{k\delta}^{\vare}}\right)d\bar{W}^{Q_2} _r \nonumber\\
=\!\!\!\!\!\!\!\!&&\hat Y_{k\delta}^{\vare}
+\frac{1}{\vare}\int_{0}^{s}AY_{\frac{r}{\vare}}^{X_{k\delta}^{\vare},\hat Y_{k\delta}^{\vare}}dr
+\frac{1}{\vare}\int_{0}^{s}g\left(X_{k\delta}^{\vare},Y_{\frac{r}{\vare}}^{X_{k\delta}^{\vare},\hat Y_{k\delta}^{\vare}}\right)dr\nonumber\\
&&+\frac{1}{\sqrt{\vare}}\int_{0}^{s}\sigma_2\left(X_{k\delta}^{\vare}, Y_{\frac{r}{\vare}}^{X_{k\delta}^{\vare},\hat Y_{k\delta}^{\vare}}\right)d\hat{W}^{Q_2}_r,  \label{E3.27}
\end{eqnarray}
where
$\hat{W}^{Q_2}_t:=\vare^{1/ 2}\bar{W}^{Q_2}_{\frac{t}{\vare}}$. Then the uniqueness of the solution to equations \eref{E3.26} and \eref{E3.27} implies
that the distribution of $\left(X_{k\delta}^{\vare}, \hat{Y}^{\vare}_{s+k\delta}\right)_{0\leq s\leq \delta}$
coincides with the distribution of
$\left(X_{k\delta}^{\vare}, Y_{\frac{s}{\vare}}^{X_{k\delta}^{\vare},\hat Y_{k\delta}^{\vare}}\right)_{0\leq s\leq \delta}$.

\vspace{2mm}
Now we try to control  $\|J_1^{\vare}(t)\|_{1/2}$:
\begin{eqnarray}
\!\!\!\!\!\!\!\!&& \EE\left(\sup_{t\in[0,T]} \left\| J_1^{\vare}(t) \right\|^{2}_{1/2}\right) \nonumber\\
=\!\!\!\!\!\!\!\!&&\EE\left\{\sup_{t\in[0,T]}\left |\sum_{k=0}^{[t/\delta]-1}e^{(t-(k+1)\delta)A}
\int_{k\delta}^{(k+1)\delta}(-A)^{1/4}e^{((k+1)\delta-s)A}\left[f\left(X_{k\delta}^{\vare},\hat{Y}_{s}^{\vare}\right)
- \bar{f}\left(X_{k\delta}^{\vare}\right)\right]ds\right |^{2}\right\}\nonumber\\
\leq\!\!\!\!\!\!\!\!&& \EE\left\{\sup_{t\in[0,T]}\left[[t/\delta]\sum_{k=0}^{[t/\delta]-1}\Big |\int_{k\delta}^{(k+1)\delta}
(-A)^{1/4}e^{((k+1)\delta-s)A}\left[f\left(X_{k\delta}^{\vare},\hat{Y}_{s}^{\vare}\right)-\bar{f}
\left(X_{k\delta}^{\vare}\right)\right]ds\Big |^{2}\right]\right\}\nonumber\\
\leq\!\!\!\!\!\!\!\!&&\left[\frac{T}{\delta}\right]\sum_{k=0}^{\left[\frac{T}{\delta}\right]-1}
\mathbb{E} \Big |\int_{k\delta}^{(k+1)\delta}(-A)^{1/4}e^{((k+1)\delta-s)A}\left[f\left(X_{k\delta}^{\vare},\hat{Y}_{s}^{\vare}\right)
-\bar{f}\left(X_{k\delta}^{\vare}\right)\right]ds\Big|^{2}\nonumber\\
\leq\!\!\!\!\!\!\!\!&&\frac{C_{T}}{\delta^{2}}\max_{0\leq k\leq\left[\frac{T}{\delta}\right]-1}\mathbb{E}
\Big | \int_{k\delta}^{(k+1)\delta}(-A)^{1/4}e^{((k+1)\delta-s)A}\left[f\left(X_{k\delta}^{\vare},\hat{Y}_{s}^{\vare}\right)-\bar{f}\left(X_{k\delta}^{\vare}\right)\right]ds \Big |^{2}  \nonumber\\
=\!\!\!\!\!\!\!\!&&
C_{T}\frac{\vare^{2}}{\delta^{2}}\max_{0\leq k\leq\left[\frac{T}{\delta}\right]-1}
\mathbb{E}
\Big| \int_{0}^{\frac{\delta}{\vare}}
(-A)^{1/4}e^{(\delta-s\vare)A}
\left[f\left(X_{k\delta}^{\vare},\hat{Y}_{s\de+k\delta}^{\vare}\right)-\bar{f}\left(X_{k\delta}^{\vare}\right)\right]ds\Big|^{2}  \nonumber\\
=\!\!\!\!\!\!\!\!&&2C_{T}\frac{\vare^{2}}{\delta^{2}}\max_{0\leq k\leq\left[\frac{T}{\delta}\right]-1}\int_{0}^{\frac{\delta}{\vare}}
\int_{r}^{\frac{\delta}{\vare}}\Psi_{k}(s,r)dsdr,  \nonumber
\end{eqnarray}
where
\begin{eqnarray*}
\Psi_k(s,r):=\mathbb{E}&&\!\!\!\!\!\!\!\!\Big\langle(-A)^{1/4}e^{(\delta-s\varepsilon)A}\left[f(X_{k\delta}^\varepsilon,\hat{Y}^\varepsilon_{s\varepsilon+k\delta})-\bar{f}(X_{k\delta}^\varepsilon)\right],
\\\!\!\!\!\!\!\!\!&&
(-A)^{1/4}e^{(\delta-r\varepsilon)A}\left[f(X_{k\delta}^\varepsilon,\hat{Y}^\varepsilon_{r\varepsilon+k\delta})-\bar{f}(X_{k\delta}^\varepsilon)\right]\Big\rangle.\nonumber
\end{eqnarray*}
Denote the $\sigma$-field generated by $\{Y_u^{x,y};u\leq s\}$ by
\begin{eqnarray*}
\tilde{\mathcal{F}}_s:=\sigma(Y_u^{x,y};u\leq s).
\end{eqnarray*}
Then for $0\leq r< s\leq \frac{\delta}{\vare}$, notice that the distribution of $\left(X_{k\delta}^{\vare}, \hat{Y}^{\vare}_{s+k\delta}\right)_{0\leq s\leq \delta}$
coincides with the distribution of $\left(X_{k\delta}^{\vare}, Y_{\frac{s}{\vare}}^{X_{k\delta}^{\vare},\hat Y_{k\delta}^{\vare}}\right)_{0\leq s\leq \delta}$, we can get that
\begin{eqnarray*}
\Psi_{k}(s,r)
=\!\!\!\!\!\!\!\!&&\mathbb{E}\Big\{\EE\left\langle (-A)^{1/4}e^{(\delta-s\vare)A}
\big[f\left(x,Y_{s}^{x,y}\right)-\bar{f}(x)\big],
\right.\\
&&\quad\quad\left.
 (-A)^{1/4}e^{(\delta-r\vare)A}\left[f\left(x, Y_{r}^{x,y}\right)-\bar{f}(x)\right]\right\rangle \mid_{(x,y)=\left(X_{k\delta}^{\vare},\hat{Y}^{\vare}_{k\delta}\right)}\Big\}  \nonumber\\
=\!\!\!\!\!\!\!\!&&\mathbb{E}\left\{\EE\Big[\left\langle (-A)^{1/4}e^{(\delta-s\vare)A}
\EE \big[f\left(x,Y_{s}^{x,y}\right)-\bar{f}(x)\mid \tilde {\mathcal{F}}_{r}\big], \right.\right.\\
&&\quad\quad\left.\left.(-A)^{1/4}e^{(\delta-r\vare)A}\big[f\left(x, Y_{r}^{x,y}\right)-\bar{f}(x)\big]\right\rangle\Big]\mid_{(x,y)=\left(X_{k\delta}^{\vare},\hat{Y}^{\vare}_{k\delta}\right)}\right\}
\end{eqnarray*}
Then by the Markov property of $Y^{x,y}_t$, \eref{SP} and Theorem \ref{Rem 4.1}, we have
\begin{eqnarray*}
\Psi_{k}(s,r)
\leq\!\!\!\!\!\!\!\!&&C(\delta-s\vare)^{-1/4}(\delta-r\vare)^{-1/4}
\nonumber\\
\!\!\!\!\!\!\!\!&& \cdot\mathbb{E}\left\{\EE\left[\left|\EE f\left(x,Y_{s-r}^{x,z}\right)-\bar{f}(x)\right|\mid_{\{z=Y^{x,y}_r\}}\left(1+|x|+|Y_{r}^{x,y}|\right)\right]\mid_{(x,y)=\left(X_{k\delta}^{\vare},\hat Y_{k\delta}^{\vare}\right)}\right\}\nonumber\\
\leq\!\!\!\!\!\!\!\!&& C(\delta-s\vare)^{-1/4}(\delta-r\vare)^{-1/4}e^{-(s-r)\eta} \mathbb{E}\left[\EE\left(1+|x|^2+\left|Y^{x,y}_{r}\right|^2\right)\mid_{(x,y)=\left(X_{k\delta}^{\vare},\hat Y_{k\delta}^{\vare}\right)}\right]\nonumber\\
\leq\!\!\!\!\!\!\!\!&&C\mathbb{E}\left(1+\left|X_{k\delta}^{\vare}\right|^2+\left|\hat Y^{\vare}_{k\delta}\right|^2\right)(\delta-s\vare)^{-1/4}(\delta-r\vare)^{-1/4}e^{-(s-r)\eta}\nonumber\\
\leq\!\!\!\!\!\!\!\!&&C_T\left(1+ |x|^2+ |y|^2\right)(\delta-s\vare)^{-1/4}(\delta-r\vare)^{-1/4}e^{-(s-r)\eta},
\end{eqnarray*}
where the last two inequalities are deduced by \eref{PMY} and \eref{3.13a}. Then we get
\begin{eqnarray*}
\mathbb{E}\left[\sup_{t\in[0,T]}\|J_1^\varepsilon(t)\|_{1/2}^2\right]
&&\!\!\!\!\!\!\!\!\leq C_T(1+|x|^2+|y|^2)\frac{\varepsilon^2}{\delta^2}
\\\!\!\!\!\!\!\!\!&&\cdot\int_0^\frac{\delta}{\varepsilon}\int_r^\frac{\delta}{\varepsilon}(\delta-s\varepsilon)^{-1/4}
(\delta-r\varepsilon)^{-1/4}e^{-(s-r)\eta}dsdr\\
&&\!\!\!\!\!\!\!\!\leq C_T(1+|x|^2+|y|^2)\frac{\varepsilon^2}{\delta^2}
\\\!\!\!\!\!\!\!\!&&
\cdot\int_0^\frac{\delta}{\varepsilon}\left[\int_r^\frac{\delta}{\varepsilon}(\delta-s\varepsilon)^{-1/2}ds\right]^{1/2}
\left[\int_r^\frac{\delta}{\varepsilon}e^{-2s\eta}ds\right]^{1/2}(\delta-r\varepsilon)^{-1/4}e^{r\eta}dr\\
&&\!\!\!\!\!\!\!\!\leq C_T(1+|x|^2+|y|^2)\frac{\varepsilon^2}{\delta^2}\int_0^\frac{\delta}{\varepsilon}\frac{1}{\varepsilon^{1/2}}(\delta-r\varepsilon)^{1/4}e^{-r\eta}
(\delta-r\varepsilon)^{-1/4}e^{r\eta}dr\\
&&\!\!\!\!\!\!\!\!\leq  C_T(1+|x|^2+|y|^2)\frac{\varepsilon^{1/2}}{\delta},
\end{eqnarray*}
which completes the estimate \eref{3.23}. The proof is complete.
\end{proof}

\subsection{Proof of Theorem \ref{main result 1}:}
Taking $\delta=\vare^{\frac{1}{3}}$, we can get by Lemmas $\ref{DEX}$ and $\ref{ESX}$ that
\begin{eqnarray}\label{9.1}
\mathbb{E}\left(\sup_{t\in [0, T]}|X^{\vare}_{t}-\bar{X}_{t}|^2 1_{\{T\leq\tilde{\tau}^{\vare}_{R}\}}\right)
\leq\!\!\!\!\!\!\!\!&&C\mathbb{E}\left(\sup_{t\in [0, T\wedge\tilde{\tau}_{R}^{\vare}]}|X^{\vare}_{t}-\hat{X}^{\vare}_{t}|^2
+\sup_{t\in [0, T\wedge\tilde{\tau}_{R}^{\vare}]}|\hat X^{\vare}_{t}-\bar{X}_{t}|^2\right)\nonumber\\
\leq\!\!\!\!\!\!\!\!&&C_{R,T}(1+|x|^3+|y|^{3})\left(\frac{\varepsilon^\frac{1}{2}}{\delta}+\delta^{1/2}\right)\nonumber\\
\leq\!\!\!\!\!\!\!\!&&C_{R,T}(1+|x|^3+|y|^{3}) \vare^{\frac{1}{6}}.
\end{eqnarray}
By chebyshev's inequality, Lemmas \ref{PMY}, \ref{MDYa} and \ref{L3.8}, we have
\begin{eqnarray}
&&\mathbb{E}\left(\sup_{t\in [0, T]}|X_{t}^{\vare}-\bar{X}_{t}|^2 1_{\{T>\tilde{\tau}^{\vare}_{R}\}}\right)\nonumber\\
\leq\!\!\!\!\!\!\!\!&&\left[\mathbb{E}\left(\sup_{t\in [0, T]}|X_{t}^{\vare}-\bar{X}_{t}|^{4}\right)\right]^{1/2}
\cdot\left[\mathbb{P}\left(T>\tilde{\tau}_{R}^{\vare}\right)\right]^{1/2} \nonumber\\
\leq\!\!\!\!\!\!\!\!&& \frac{C_{T}(1+|x|^2+|y|^{2})}{\sqrt{R}}\left[\mathbb{E}\left(\int_0^T\|\bar{X}_s\|_1^2ds+\int_0^T\|X^{\vare}_s\|_1^2ds+\int_0^T\|\hat{X}_s^\varepsilon\|_1^2ds+\int_0^T\|J_s^\varepsilon\|_{1/2}^2ds\right)\right]^{1/2}\nonumber\\
\leq\!\!\!\!\!\!\!\!&&\frac{C_{T}(1+|x|^3+|y|^{3})}{\sqrt{R}}. \label{after tau_R}
\end{eqnarray}
Hence, combining \eref{9.1} and \eref{after tau_R} together yields
\begin{eqnarray*}
\mathbb{E}\left(\sup_{t\in [0, T]}|X_{t}^{\vare}-\bar{X}_{t}|^2\right)\leq\!\!\!\!\!\!\!\!&& C_{R,T}(1+|x|^3+|y|^{3}) \vare^{\frac{1}{6}}+\frac{C_{T}(1+|x|^3+|y|^{3})}{\sqrt{R}}.
\end{eqnarray*}
Now, letting $\vare\rightarrow 0$ firstly and then $R\rightarrow \infty$, we have
\begin{eqnarray}
\lim_{\vare\rightarrow 0}\mathbb{E}\left(\sup_{t\in [0, T]}|X_{t}^{\vare}-\bar{X}_{t}|^2\right)=0.\label{L2}
\end{eqnarray}
Then for any $p\geq 1$, we have
\begin{eqnarray*}
\EE\left(\sup_{t\in [0, T]}|X_{t}^{\vare}-\bar{X}_{t}|^{2p}\right)=\!\!\!\!\!\!\!\!&& \EE\left[\sup_{t\in [0, T]}\left(|X_{t}^{\vare}-\bar{X}_{t}| |X_{t}^{\vare}-\bar{X}_{t}|^{2p-1}\right)\right]\\
\leq\!\!\!\!\!\!\!\!&&C_p\left[\EE\left(\sup_{t\in [0, T]}|X_{t}^{\vare}-\bar{X}_{t}|^{2}\right)\right]^{1/2}\left[\EE\left(\sup_{t\in [0, T]}|X_{t}^{\vare}-\bar{X}_{t}|^{4p-2}\right)\right]^{1/2}.
\end{eqnarray*}
As a result, by \eref{L2}, Lemmas \ref{PMY} and \ref{L3.8}, it is easy to see
$$
\lim_{\vare\rightarrow 0}\EE\left(\sup_{t\in [0, T]}|X_{t}^{\vare}-\bar{X}_{t}|^{2p}\right)=0,
$$
which completes the proof.\hspace{\fill}$\Box$

\section{Appendix} \label{Sec appendix}

The following properties of $b(\cdot,\cdot,\cdot)$ and $B(\cdot,\cdot)$ are well-known (for example see \cite{DX,FMRT,HM1}). 
\begin{lemma} \label{Property B1}
For any $u,v,w\in H^1$, we have
$$ b(u,v,w)=-b(u,w,v),\quad b(u,v,v)=0.$$
Moreover, if $\alpha_{i}\geq 0~(i=1,2,3)$ satisfies one of the following conditions \\
$(1) ~\alpha_{i}\neq1(i=1,2,3), \alpha_{1}+\alpha_{2}+\alpha_{3}\geq 1$, \\
$(2) ~\alpha_{i}=1$ for some $i$, $\alpha_{1}+\alpha_{2}+\alpha_{3}>1$,\\
then $b(u,v,w)$ is continuous from
${H}^{\alpha_{1}}\times
{H}^{\alpha_{2}+1}\times
{H}^{\alpha_{3}}$ to $\mathbb{R}$, i.e.
\begin{eqnarray*}|b(u,v,w)|\leq C\|u\|_{\alpha_{1}}\|v\|_{\alpha_{2}+1}\|w\|_{\alpha_{3}}.\end{eqnarray*} 
\end{lemma}
From the definition of $b$, interpolation inequality and Lemma \ref{Property B1} above, the following inequalities can be
derived.
\begin{corollary} \label{Property B2} For any $u, v\in {H}^1$, there exists a constant $C>0$ such that
\begin{enumerate}
\item[(1)]$\langle B(u),v\rangle\leq C\|u\|_1^{3/2}|u|^{1/2}|v|_{L^4};$
\item[(2)]$\|B(u)\|_{-1}\leq C|u|\|u\|_{1};$
\item[(3)] $\langle B(u)-B(v),u-v\rangle\leq C|u-v|\|u-v\|_1\|u\|_1;$
\item[(4)]  $\|B(u)-B(v)\|_{-1}\leq C\|u-v\|_{1/2}\left(\|u\|_{1/2}+\|v\|_{1/2}\right).$
\end{enumerate}
\end{corollary}

\medskip
At the end of this section, we give the proof of Theorem \ref{Th1} based on the techniques used in \cite{DM}.

\medskip
\noindent
\textbf{Proof of Theorem \ref{Th1}:}
Let $\mathcal{H}:={H}\times H$ be the product Hilbert space. For any $\phi=(\phi_1,\phi_2),\varphi=(\varphi_1,\varphi_2)\in\mathcal{H}$, we denote the scalar product and the induced norm by
\begin{eqnarray*}
\langle\phi,\varphi\rangle_{\mathcal{H}}=\int_{\mathbb{T}^2}\phi_1({\xi})\varphi_1({\xi})d{\xi}+\int_{\mathbb{T}^2}\phi_2({\xi})\varphi_2({\xi})d{\xi},~~
|\phi|_{\mathcal{H}}=\sqrt{\langle\phi,\phi\rangle_{\mathcal{H}}}=\sqrt{|\phi_1|^2+|\phi_2|^2}.
\end{eqnarray*}
Similarly, we also define $\mathcal{V}:={H}^1\times H^1$. Then $\mathcal{V}$ is a product Hilbert space with the scalar product and the induced norm,
\begin{eqnarray*}&&\langle\phi,\varphi\rangle_{\mathcal{V}}=\int_{\mathbb{T}^2}\nabla\phi_1({\xi})\cdot\nabla\varphi_1({\xi})d{\xi}
+\int_{\mathbb{T}^2}\nabla\phi_2({\xi})\cdot\nabla\varphi_2({\xi})d{\xi},
\\
&&\|\phi\|_{\mathcal{V}}=\sqrt{\langle\phi,\phi\rangle_{\mathcal{V}}}=\sqrt{\|\phi_1\|_1^2+\|\phi_2\|_1^2}.\end{eqnarray*}
Now we rewrite the system \eref{main equation} for $Z^{\varepsilon}=(X^{\varepsilon},Y^{\varepsilon})$ as
\begin{eqnarray}
dZ^{\varepsilon}=\tilde{A}Z^{\varepsilon}dt+F(Z^{\varepsilon})dt+\sigma(Z^{\varepsilon})dW_t,\quad Z^{\varepsilon}_0=(x,y)\in \mathcal{H},\label{E4.1}
\end{eqnarray}
where $W_t:=(W_t^{Q_1},W_t^{Q_2})$, which is a $\mathcal{H}$-valued $Q$-Wiener process with $Q:=(Q_1, Q_2)$ and $Q$ is a positive symmetric, trace class operate on $\mathcal{H}$ and
\begin{eqnarray*}
&&\tilde{A}Z^{\varepsilon}=\left(AX^{\varepsilon},\frac{1}{\varepsilon}AY^{\varepsilon}\right),
\\&&F(Z^{\varepsilon})=\left(-B(X^{\varepsilon})+f(X^{\varepsilon},Y^{\varepsilon}),\frac{1}{\varepsilon}g(X^{\varepsilon},Y^{\varepsilon})\right),
\\&&\sigma(Z^{\varepsilon})=\left(\sigma_1(X^{\varepsilon}),\frac{1}{\sqrt{\varepsilon}}\sigma_2(X^{\varepsilon},Y^{\varepsilon})\right).
\end{eqnarray*}
Moreover, $\sigma$ is an operator from $\mathcal{H}\rightarrow \mathcal{L}_{Q}(Q^{1/2}\mathcal{H},\mathcal{H})$, here $\mathcal{L}_{Q}(Q^{1/2}\mathcal{H},\mathcal{H})$ is the space of linear operate $S:Q^{1/2}\mathcal{H}\rightarrow \mathcal{H}$ such that
$SQ^{1/2}$ is a Hilbert-Schmidit operator from $\mathcal{H}$ to $\mathcal{H}$. The norm in $\mathcal{L}_{Q}(Q^{1/2}\mathcal{H},\mathcal{H})$ is defined by
$$|\sigma(z)|_{\mathcal{L}_{Q}}=\sqrt{|\sigma_1(x)|_{\mathcal{L}_{Q_1}}^2+\frac{1}{\varepsilon}|\sigma_2(x, y)|_{\mathcal{L}_{Q_2}}^2},\quad z=(x,y)\in\mathcal{H}.$$
Let $\mathcal{V}^{'}$
be the dual space of $\mathcal{V}$ and we consider the following  Gelfand triple $\mathcal{V}\subset\mathcal{H}\equiv\mathcal{H}^{'}\subset\mathcal{V}^{'}$.
It is easy to see that the following mappings
$$\tilde{A}:\mathcal{V}\rightarrow\mathcal{V}^{'},~~F:\mathcal{V}\rightarrow\mathcal{V}^{'}$$
are well defined. To complete the proof, it remains to take finite-dimensional Galerkin approximation, and apply the weak convergence approach. The details are similar to the arguments in \cite{DM,MS}, where the authors deal with the 2D stochastic Boussinesq equation and 2D stochastic Navier-Stokes equation, respectively. So we only check the new coefficients in equation \eref{E4.1} satisfy the local monotonicity and coercivity properties and further details are omitted here.

Indeed,
for any $w=(u,v)\in\mathcal{V}$, by Lemma \ref{Property B1}, there exist constants $C_{\vare}>0$ and $C>0$ such that
\begin{eqnarray*}
\langle\tilde{A}w+F(w) ,w\rangle+|\sigma(w)|^2_{\mathcal{L}_{Q}}=\!\!\!\!\!\!\!\!&&\langle{A}u-B(u)+f(u,v),u\rangle
+\frac{1}{\varepsilon}\langle Av+g(u,v),v\rangle
\\\!\!\!\!\!\!\!\!&&+C(1+|u|^2+|v|^2)
\\\leq\!\!\!\!\!\!\!\!&&-\|u\|^2_1-\frac{1}{\varepsilon}\|v\|^2_1+C(1+|u|^2+|v|^2)
\\\leq\!\!\!\!\!\!\!\!&&-C_{\vare}\|w\|^2_{\mathcal{V}}+C(1+|w|^2_{\mathcal{H}}),
\end{eqnarray*}
which implies that the coercivity condition holds. For any $w_1=(u_1,v_1),w_2=(u_2,v_2)\in\mathcal{V}$, by Corollary \ref{Property B2} and the condition \ref{A1}, we have
\begin{eqnarray*}
&&\langle F(w_1)-F(w_2),w_1-w_2\rangle
\\=\!\!\!\!\!\!\!\!&&\langle -B(u_1-u_2,u_2)+f(u_1,v_1)-f(u_2,v_2),u_1-u_2\rangle
+\frac{1}{\varepsilon}\langle g(u_1,v_1)-g(u_2,v_2),v_1-v_2\rangle
\\\leq\!\!\!\!\!\!\!\!&&2\|u_1-u_2\|_1^{3/2}|u_1-u_2|^{1/2}| u_2|_{L^4}+C(|u_1-u_2|^2+|v_1-v_2|^2)
\\\leq\!\!\!\!\!\!\!\!&&\frac{1}{2}\|u_1-u_2\|_1^2+C(1+| u_2|_{L^4}^4)|u_1-u_2|^2+C|v_1-v_2|^2.
\end{eqnarray*}
Therefore, we get
\begin{eqnarray*}
&&\langle \tilde{A}w_1+F(w_1)-\tilde{A}w_2-F(w_2),w_1-w_2\rangle+|\sigma(w_1)-\sigma(w_2)|^2_{\mathcal{L}_{Q}}\\
\leq\!\!\!\!\!\!\!\!&&-|u_1-u_2|^2_1-\frac{1}{\vare}|v_1-v_2|^2_1+\frac{1}{2}\|u_1-u_2\|_1^2+C(1+| u_2|_{L^4}^4)|u_1-u_2|^2+C|v_1-v_2|^2\\
\leq\!\!\!\!\!\!\!\!&&C(1+|u_2|_{L^4}^4)|w_1-w_2|^2_{\mathcal{H}}.
\end{eqnarray*}
\hspace{\fill}$\Box$

\vspace{0.3cm}
\textbf{Acknowledgment}. This work is supported by the NNSF of China (No.: 11601196, 11771187), the Natural Science Foundation of the Jiangsu Higher Education Institutions of China (16KJB110006) and the PAPD of Jiangsu Higher Education Institutions.

\vspace{0.3cm}

\end{document}